\documentclass{article}
\usepackage[british]{babel}
\usepackage[utf8]{inputenc}
\usepackage{hyperref,xcolor}
\usepackage{caption,subcaption,longtable}
\usepackage{graphics, setspace}
\usepackage{mathtools}
\usepackage{array,multirow}
\usepackage{tikz,float}
\usepackage{amsmath,amsthm,amssymb}
\usepackage{enumerate,enumitem,listings}
\usepackage[T1]{fontenc}
\usepackage{stackengine,scalerel}
\usepackage{hyphenat}
\usepackage{slantsc}
\allowdisplaybreaks
\newcommand{\mathkey}{
\stackon[0 pt]{$\Delta$}{$\boldsymbol{\circ}$}}
\newcommand{\Gw}{\Gamma(m,\underline{w})}
\newcommand{\Gpp}{\Gamma(m,\pi,p,*)}
\newcommand{\Bpp}{\mathcal{B}(m_{1},m_{2},\pi,p)}
\newcommand{\Dpp}{\mathkey(m_{1},m_{2},\pi,p,*)}
\newcommand{\Lpp}{\Lambda(m_{1},m_{2},\pi,p,*)}
\newcommand{\MGpp}{M(\Gamma(m,\pi,p,*))}
\newcommand{\MBpp}{M(\mathcal{B}(m_{1},m_{2},\pi,p))}
\newcommand{\MDpp}{M(\mathkey(m_{1},m_{2},\pi,p,*))}
\newcommand{\MLpp}{M(\Lambda(m_{1},m_{2},\pi,p,*))}
\newcommand{\norm}[1]{\vert \vert \underline{\boldsymbol{#1}}\vert \vert }
\newcommand{\vect}[1]{\underline{\boldsymbol{#1}}}
\newcommand{\RG}[1]{\mathbb{R}[#1]}
\newcommand{\I}[1]{I[#1]}

\makeatletter
\def\@maketitle{%
  \newpage
  \begin{center}%
  \let \footnote \thanks
    {\LARGE \@title \par}%
    \vskip 1.5em%
    {\large
      \lineskip .5em%
      \begin{tabular}[t]{c}%
        \@author
      \end{tabular}\par}%
    \vskip 1em%
    {\large \@date}%
  \end{center}%
  \par
  \vskip 1.5em}
\makeatother

\makeatletter
\newtheorem*{rep@theorem}{\rep@title}
\newcommand{\newreptheorem}[2]{%
\newenvironment{rep#1}[1]{%
 \def\rep@title{#2 \ref{##1}}%
 \begin{rep@theorem}}%
 {\end{rep@theorem}}}
\makeatother
\newtheorem{theoremalph}{Theorem}

\newtheorem{theorem}{Theorem}[section]
\newtheorem{lemma}[theorem]{Lemma}

\newtheorem{cor}[theorem]{Corollary}
\newtheorem*{theorem*}{Theorem}
\newtheorem*{lemma*}{Lemma}
\newtheorem{question}{Question}
\newtheorem*{question*}{Question}
\newtheorem*{proposition*} {Proposition}
\newreptheorem{theorem}{Theorem}
\newreptheorem{cor}{Corollary}
\theoremstyle{definition}
\newtheorem{definition}[theorem]{Definition}
\newtheorem{remark}[theorem]{Remark}
\newtheorem*{remark*}{Remark}

\setlist[enumerate]{align=left}
\allowdisplaybreaks
\title{Property (T) in random quotients of hyperbolic groups at densities above $1\slash 3$}
\author{Calum J. Ashcroft}
\date{\vspace{-2em}}

\begin{document}

\maketitle
	\begin{abstract}
We study random quotients of a fixed non-elementary hyperbolic group in the Gromov density model. Let $G=\langle S\;\vert\; T\rangle $ be a finite presentation of a non-elementary hyperbolic group, and let $Ann_{l,\omega }(G)$ be the set of elements of norm between $l-\omega(l)$ and $l$ in $G$. A random quotient at density $d$ and length $\omega$-near $l$ is defined by killing a uniformly randomly chosen set of $\vert S_{l}(G)\vert ^{d}$ words in $Ann_{l,\omega (l)}(G)$, where $\omega (l) =o_{l}(l)$. We prove that for any $d>1\slash 3$, such a quotient has Property (T) with probability tending to $1$ as $l$ tends to infinity. This result answers a question of Gromov--Ollivier and strengthens a theorem of \.{Z}uk (c.f Kotowski--Kotowski).
	\end{abstract}

%----------------------------------------------------------------------------------------
%   Introduction
%----------------------------------------------------------------------------------------
%\input{Sections/introduction}

\section{Introduction}

The study of random quotients of hyperbolic groups goes back to the origins of geometric group theory. Indeed, Gromov posed a question concerning Property (T) in such quotients in his original monograph \cite[$\S$4.5C]{Gromov_hyperbolic}. This question was later refined by Ollivier in \cite[$\S$IV.c.]{Ollivier_Jan_Invitation}. Let $G=\langle S\;\vert \; T\rangle $ be a finite presentation of a non-elementary hyperbolic group with growth rate $\mathfrak{h},$ so that letting $S_{l}(G)$ be the sphere of radius $l$ in $G$, we have $\vert S_{l}(G)\vert =\eta(1+o_{l}(1))\exp\{\mathfrak{h}l\}$ for some constant $\eta=\eta (G) >0$ \cite{Cannon_combinatorial_structure_hyperbolic_groups} (see e.g. \cite{calegari2013ergodic} for an overview). Since hyperbolic groups are coarse-geometric objects, it is most natural to study quotients at length `near' $l$; see e.g. \cite[\S 1.2C]{Ollivier_Jan_Invitation}.
\begin{question*}[Gromov--Ollivier]
Does a random quotient of a non-elementary hyperbolic group $G$ at density $d>1\slash 3$ have Property (T) with probability tending to $1$ as $l$ tends to infinity?
\end{question*}
The above question was answered in the case of free groups by a major result of \.{Z}uk \cite{Zuk} (c.f. \cite{kotowski, DRUTU_Mackay, ashcroft2021property}). This work provided a large class of hyperbolic groups with (T), a previously rare phenomenon. However, one of the major advances of geometric group theory when compared to classical combinatorial group theory is the ability to finely control the quotients of hyperbolic groups, as opposed to understanding only the quotients of free groups. It is therefore still of great interest to answer the Gromov--Ollivier question; we prove the following, answering the Gromov--Ollivier question in the affirmative. 

\begin{theoremalph}\label{mainthm: property t in random quotients of hyperbolic groups}
Let $G=\langle S\;\vert \; T\rangle $ be a non-elementary hyperbolic group with growth rate $\mathfrak{h}.$ Let $\omega(l)\leq \log\log l$ be any slow-growing function. Let $d>1\slash 3$ and for each $l\geq 1$ uniformly randomly select a set $\mathcal{R}_{l}\subseteq Ann_{l,\omega}(G)$ of size $ \exp\{\mathfrak{h}ld\}$. The group $G\slash \langle\langle \mathcal{R}_{l}\rangle\rangle$ has Property (T) with probability tending to $1$ as $l$ tends to infinity.
\end{theoremalph}

A function $\omega$ is \emph{slow-growing} if both $\omega(n+1)\slash \omega(n)\rightarrow 1$ and $\omega (n) \rightarrow \infty$ as $n\rightarrow \infty.$ The main difficulty in the Gromov--Ollivier question stems from the fact that the Markov chain, $M$, parametrizing the automatic structure for $G$ will often not be connected, let alone ergodic. To circumvent this, we use a specific connected component $C^{max}$ of $M$ that arises in the work of \cite{gekhtman2018counting}. If we assume that $C^{max}$ is aperiodic, we may prove the following, stronger, theorem.

\begin{theoremalph}\label{mainthm: property t in random quotients of hyperbolic groups aperiodic case}
Let $G=\langle S\;\vert \; T\rangle $ be a non-elementary hyperbolic group with growth rate $\mathfrak{h}.$ Let $d>1\slash 3$ and for each $l\geq 1$ uniformly randomly select a set $\mathcal{R}_{l}\subseteq Ann_{l,2}(G)$ of size $\exp\{\mathfrak{h}ld\}$. If $C^{max}$ is aperiodic, then the group $G\slash \langle\langle \mathcal{R}_{l}\rangle\rangle$ has Property (T) with probability tending to $1$ as $l$ tends to infinity.
\end{theoremalph}

Clearly, Theorems \ref{mainthm: property t in random quotients of hyperbolic groups} and \ref{mainthm: property t in random quotients of hyperbolic groups aperiodic case} hold for any constant multiple of $\exp\{\mathfrak{h}ld\}$.
Since the Markov chain $M$ is connected and aperiodic for free groups, Theorem \ref{mainthm: property t in random quotients of hyperbolic groups aperiodic case}  strengthens the result of \cite{Zuk, kotowski} (c.f. \cite{DRUTU_Mackay}), which considers random quotients of free groups, i.e. the Gromov--Ollivier question for free groups only. \footnote{Ollivier's question for free groups concerns random quotients at length $l$ and discusses elements of length $l\slash 3$, so clearly requires $l$ to be divisible by $3$. There is no distinction in the literature between the case of random quotients at exactly length $3l$ versus quotients at length near $l$, with function $\omega (l)=2$. It should be possible to extend Theorem \ref{mainthm: property t in random quotients of hyperbolic groups aperiodic case} to choices of random subsets $R_{l}\subseteq S_{l}(G)$ in a manner similar to the extension of \.{Z}uk's theorem in \cite{ashcroft2021property}. Due to the length of this paper we have decided to omit this extension.}

Studying the random quotients of a general hyperbolic group is much more complicated than studying the random quotients of a free group. One of the main difficulties is that the sphere of radius $l$ in $F_{n}$ is far more easily understood than that of a more general hyperbolic group. This difficulty occurs as the Markov chain parameterizing geodesics is irreducible for free groups, but will often be reducible for hyperbolic groups. Detailed study of this Markov chain can be found in \cite{gekhtman2018counting} and \cite{gekhtman2020central}. A further difficulty of the Gromov--Ollivier question stems from the fact that we are quotienting by elements of the group, rather than words in the generating set. In the latter case, this would give us a quotient of a random quotient of $F_{n}$, and so we would be able to apply the work of \.{Z}uk \cite{Zuk} (c.f. \cite{kotowski}) to deduce (T). However, if the group $G$ has cogrowth $\xi$ (which necessarily satisfies $1\geq \xi\geq 1\slash 2$), and $d>1-\xi$, then such a random quotient is trivial with high probability \cite[Theorem 2]{olliviersharp}. We could therefore meet situations where $1\slash 3<1-\xi$, and so all such quotients would be trivial with high probability.

Importantly, the \emph{only} requirement we place on $G$ is that it is non-elementary hyperbolic. In particular, $G$ is allowed to contain torsion. We believe this to be the only result in the literature that considers the random quotients of \emph{any} non-elementary hyperbolic group. To our knowledge, the following two results are the only prior results in the literature studying random quotients of a class of non-elementary hyperbolic groups that is not limited to contain only free groups. In \cite{olliviersharp}, Ollivier considered quotienting non-elementary hyperbolic groups with `harmless torsion' by a set of $ \exp\{\mathfrak{h}ld\}$ words chosen randomly with respect to a sequence of measures $\{\mu_{l}\}_{l}$. In the case that the quotients are taken with respect to the uniform measure on elements of $S_{l}(G)$, then such a quotient is infinite, torsion-free, and hyperbolic with probability tending to $1$ if $d<1\slash 2$ and is trivial if $d>1\slash 2$ \cite[Theorem 3]{olliviersharp} (this is a specific case of the far more general result proved in the same paper; see \cite[Theorem 13]{olliviersharp}). Let us suppose instead that $G=\pi_{1}(X)$ for a non-positively curved cube complex, and take a random quotient of $G$ by a uniformly random set of $ \exp\{\mathfrak{h}ld\}$ words in $S_{l}(G)$. There exists a $d_{cub}=d_{cub}(\langle S\;\vert\;T\rangle)$ such that for $d<d_{cub}$, a random quotient at density $d$ is also the fundamental group of a non-positively curved cube complex with probability tending to $1$ \cite{futer2021cubulating}. Of course, Futer--Wise's theorem cannot hold for all hyperbolic groups $G$. If $G$ has (T), then this is inherited by all quotients of $G$, so that quotients of $G$ are the fundamental group of a non-positively curved cube complex if and only if they are finite. There are, to our knowledge, no previous results on Property (T) in random quotients of a hyperbolic group that is not the free group $F_{n}$.

The density bound in Theorem \ref{mainthm: property t in random quotients of hyperbolic groups} is clearly not optimal for every hyperbolic group $G$. If the group $G$ has $(T)$, then so does any of its quotients. We note that by Ollivier, for a group $G$ with harmless torsion, and for $d<1\slash 2$, the above quotients are also infinite, non-elementary hyperbolic, and torsion free with high probability \cite{olliviersharp}. Therefore, Theorem \ref{mainthm: property t in random quotients of hyperbolic groups}, coupled with \cite{olliviersharp}, provides an interesting class of group presentations with Property (T). However, we cannot apply the work of \cite{olliviersharp} in the case that $G$ has harmful torsion. The example provided by Ollivier of the failure of \cite[Theorem 4]{olliviersharp} is not enlightening for our purposes, since the quotient is not taken with a uniform selection of words; the chosen measures $\mu_{l}$ are non-uniform. This leads us to a natural question.
\begin{question}
Let $G$ be a non-elementary hyperbolic group with harmful torsion in the sense of Ollivier, and let $1\slash 3 <d<1\slash 2$. Let $\Gamma_{l}:=G\slash \langle \langle R_{l}\rangle\rangle$ be a sequence of quotients as in Theorem \ref{mainthm: property t in random quotients of hyperbolic groups} at density $d$ (so that has $\Gamma_{l}$ has (T) almost surely). Is $\Gamma_{l}$ almost surely infinite? Is $\Gamma_{l}$ almost surely hyperbolic?
\end{question}

The classical \emph{density model} of random groups is obtained by taking $G$ to be the free group $F_{n}=\langle a_{1},\hdots ,a_{n}\;\vert\;\rangle$ and restricting relators to be cyclically reduced words. It is typical when working with models of randomness to consider asymptotics. There are many theorems concerning the properties of random groups in the density model; all of the following statements hold with probability tending to $1$. Firstly, a random group in the density model at density $d<1\slash 2$ is infinite, non-elementary hyperbolic, and torsion free \cite{gromovasymptotic} (c.f. \cite{olliviersharp}).  Groups in the density model are virtually special for $d<1\slash 6$ \cite{Ollivier-Wise, Agol13} and contain a free codimension-$1$ subgroup for $d<3\slash 14$ \cite{montee2021random} (c.f. \cite{mackay_przytycki2015balanced}). If we instead consider Property (T), then a random group in the density model has Property (T) at $d>1\slash 3$ \cite{Zuk,kotowski,ashcroft2021property} (c.f. \cite{DRUTU_Mackay}). There is a similar model (the \emph{k-gonal model}) to the density model where we keep $l$ fixed and let $n$ tend to infinity. There are results concerning hyperbolicity \cite{Zuk,odrsquaremodel,ashcroftrandom}, cubulation \cite{duong,odrcubulatingsquare,odrzygozdz2019bent}, and Property (T) \cite{Zuk,kotowski,antoniuktriangle,odrzygozdz2019bent,Montee_prop_t,ashcroft2021property} in this model.

\subsection{The structure of the paper}
The paper is structured as follows. 
 In Section \ref{sec: graph defs}, we remind the reader of some definitions relating to graphs and spectral graph theory. In Section \ref{sec: a spectral criterion for Property (T)}, we introduce a new spectral criterion for Property (T), which we obtain by an application of the work of Ozawa \cite{ozawapropertyt}. This new spectral criterion relates to an easily understood graph, $\Upsilon$, constructed from a set of relators. The edges are added to this graph independently, and so we can model $\Upsilon$ by a suitable random graph. We split each relator $r=uvw$ into $3$ almost-equal parts and add edges between $u$ and $w^{-1}$ in $\Upsilon$. The main advantage to this is that we do not need $uw$ to be geodesic. If we were to use \.{Z}uk's criterion, then we would also add the edges $(v,u^{-1})$ and $(w,v^{-1})$. Since these edges require $uv$ and $vw$ to be geodesic, this could effectively decompose $\Upsilon$ into strongly connected `cliques', with distinct cliques only connected by few edges. This could potentially decrease the first eigenvalue of $\Upsilon$ below $1\slash 2$ (the value required by \.{Z}uk's criterion \cite{zuk1996}). 
 
 The spectral criterion we introduce requires two objects to be analysed. Firstly, we are required to find a language of geodesics with certain properties. We use a result of \cite{gekhtman2018counting} on the dynamical structure of geodesics in hyperbolic groups in Section \ref{sec: counting geodesics in hyperbolic groups}, which allows us to study a set of geodesic words. Section \ref{appendix: Spectral theory of restricted graphs} studies the eigenvalues of certain random graphs that model the graph $\Upsilon$ appearing in our spectral criterion. Section \ref{sec: Property (T) in quotients of hyperbolic groups} then uses these results to provide a proof of Theorem \ref{mainthm: property t in random quotients of hyperbolic groups}. We then discuss how to adapt the proof strategy in the case that $C^{max}$ is aperiodic in order to prove Theorem \ref{mainthm: property t in random quotients of hyperbolic groups aperiodic case}.

\subsection{Some notation}
We now briefly discuss some notation and assumptions. We will fix an ordering on $S$ and choose the automatic structure for $G$ of shortlex geodesics, i.e. each element of $G$ is represented by a unique word of the same length. Let $\mathcal{L}^{geo}$ be the regular prefix-closed language of lexicographically first geodesics. For each element $g$ of $G$, we identify it with the unique word $w_{g}\in \mathcal{L}^{geo}$ that evaluates to $g$. In particular, by $S_{l}(G)$ we are referring to the elements in the sphere of radius $l$ in $G$, as well as the unique elements of $\mathcal{L}^{geo}$ evaluating to them.

We are dealing with asymptotics, and so we frequently arrive at situations where $m$ is some parameter tending to infinity that is required to be an integer. If $m$ is not integer, then we will implicitly replace it by $\lfloor m\rfloor$. Since we are dealing with asymptotics, this will not affect any of our arguments. 
\begin{definition}
Given $m_{1}:\mathbb{N}\rightarrow \mathbb{N}$ a function such that $m_{1}(m)\rightarrow\infty$ as $m\rightarrow\infty$, we write $m_{2}=m_{2}(m_{1})$ to mean that $m_{2}(m)=f(m_{1}(m))$ for some function $f$, and $m_{2}(m_{1}(m))\rightarrow\infty$ as $m\rightarrow\infty$, i.e. $m_{2}$ only depends on $m_{1}$, and tends to infinity as $m_{1}$ tends to infinity.\end{definition}
The following are standard.
\begin{definition}
   Let $f,g:\mathbb{N}\rightarrow \mathbb{R}_{+}$ be two functions. We write  $f=o(g)$ if $f(m)\slash g(m)\rightarrow 0$ as $m\rightarrow\infty$,
       $f=O(g)$ if there exists a constant $N\geq 0$ and $M\geq 1$ such that $f(m)\leq N g(m)$ for all $m\geq M$, and $f=\Omega (g)$ if $g=o(f)$. Finally, $f=\Theta (g)$ if $f(m)\slash g(m)\rightarrow c\in (0,\infty)$ as $m\rightarrow\infty.$

   We write $f=o_{m}(g)$ etc to indicate that the variable name is $m$. Typically we will deal with functions $m_{2}=m_{2}(m_{1})$, and $f=f(m_{1},m_{2})$. We will write $f=o_{m_{1}}(g)$ etc to mean that the function $f'(m_{1})=f(m_{1},m_{2}(m_{1}))=o_{m_{1}}(g'(m_{1})),$ where $g'(m_{1})=g(m_{1},m_{2}(m_{1}))$.
 % If $f=f(m,n),g=g(m,n)$ we write $f=o_{m}(g)$ if for any fixed $n$, $f(m,n)\slash g(m,n)\rightarrow 0$ as $m\rightarrow\infty$, and similarly $f=O_{m}(g),\;f=\Omega_{m}(g).$
\end{definition}

\begin{definition}
Let $\mathcal{M}(m)$ be some model of random groups (or graphs) depending on a parameter $m$, and let $\mathcal{P}$ be a property of groups (or graphs). We say that $\mathcal{P}$ holds \emph{asymptotically almost surely with} $m$ (\emph{a.a.s.}$(m)$) if 
$$\lim_{m\rightarrow\infty}\mathbb{P}(G\sim \mathcal{M}(m)\mbox{ has }\mathcal{P})=1.$$

Again, we will regularly have to deal with cases where $m_{2}=m_{2}(m_{1})$ is fixed, $\mathcal{M}(m_{1},m_{2})$ is some model of random groups (or graphs) depending on parameters $m_{1}$ and $m_{2}$, and $\mathcal{P}$ is a property of groups (or graphs). We say that $\mathcal{P}$ holds \emph{asymptotically almost surely with} $(m_{1})$ (\emph{a.a.s.}$(m_{1})$) if 
$$\lim_{m_{1}\rightarrow\infty }\mathbb{P}(G\sim \mathcal{M}(m_{1},m_{2}(m_{1}))\mbox{ has }\mathcal{P})=1.$$
\end{definition}

Finally, we will often be working with bipartite graphs. The vertex partition of a bipartite graph $\Sigma$ will always be written $V(\Sigma)=V_{1}(\Sigma)\sqcup V_{2}(\Sigma)$.

\section*{Acknowledgements}
I would like to thank Emmanuel Breuillard for discussions on random groups, as well as introducing me to the matrix Bernstein inequality and its application to random graphs, which has been of great use. I would like to thank Tom Hutchcroft for advice on references for Markov chains. I am grateful to Daniel Groves for conversations regarding the ergodic theory of encodings of geodesics in hyperbolic groups. I would like to thank Sam Taylor for pointing out that Lemma \ref{lem: Cmax gens finite index} was a result already present in the literature, as well as pointing out a mistake in the original proof of Theorem \ref{mainthm: property t in random quotients of hyperbolic groups}, which lead to the corrected statement and proof of Theorem \ref{mainthm: property t in random quotients of hyperbolic groups}, as well as the introduction of Theorem \ref{mainthm: property t in random quotients of hyperbolic groups aperiodic case}. As always, I would like to thank Henry Wilton for his advice and mathematical conversations, as well as comments on an earlier draft of this paper.
%----------------------------------------------------------------------------------------
%   Graphs and eigenvalues
%----------------------------------------------------------------------------------------

%----------------------------------------------------------------------------------------
%   Graphs and eigenvalues
%----------------------------------------------------------------------------------------
%\input{Sections/Graphs_and_eigenvalues}

\section{Graphs and eigenvalues}\label{sec: graph defs}
Let $\Sigma=(V,E)$ be a graph (we allow multiple edges between vertices, as well as loops). Given two graphs $\Sigma_{1}=(V_{1},E_{1})$ and $\Sigma_{2}=(V_{2},E_{2})$, we define the \emph{union} of $\Sigma_{1}$ and $\Sigma_{2}$ as the graph $\Sigma_{1}\cup \Sigma_{2}:=(V_{1}\cup V_{2},E_{1}\sqcup E_{2})$. In particular, we take the union of vertices, and the disjoint union of edge sets (i.e. we assume that different graphs have disjoint edge sets).

Let $\Sigma=(V,E)$ be a graph with vertex set $V=\{v_{1},\hdots ,v_{m}\}$. The \emph{adjacency matrix} of $\Sigma$, $A(\Sigma)$, is the $m\times m$ matrix with $A(\Sigma)_{i,j}$ defined to be the number of edges between $v_{i}$ and $v_{j}$. The \emph{degree matrix} of $\Sigma$, $D(\Sigma)$, is the diagonal matrix with entries $D(\Sigma)_{i,i}=deg(v_{i})$. The \emph{normalised Laplacian} of $\Sigma$, $L(\Sigma)$, is defined by $$L(\Sigma)=I-D^{-1\slash 2}AD^{-1\slash 2}.$$
We note that $L(\Sigma)$ is symmetric positive semi-definite, with eigenvalues $$0\leq \lambda_{0}(L(\Sigma))\leq \lambda_{1}(L(\Sigma))\leq \hdots \leq \lambda_{m-1}(L(\Sigma))\leq 2.$$ For $i=1,\hdots, m$, we define $\lambda_{i}(\Sigma):=\lambda_{i}(L(\Sigma))$.

We can also define $\mathcal{L}(\Sigma)=I-D(\Sigma)^{-1}A(\Sigma).$ The matrices $L(\Sigma)$ and $\mathcal{L}(\Sigma)$ are similar, so that $\lambda_{i}(L(\Sigma))=\lambda_{i}(\mathcal{L}(\Sigma))$ for each $i$. We will switch between $L(\Sigma)$ and $\mathcal{L}(\Sigma)$ depending on which best suits our purpose.

If $M$ is a symmetric real $m\times m$ matrix, then $M$ has real eigenvalues, which we order by $\lambda_{0}(M)\leq \lambda_{1}(M)\leq \hdots \leq \lambda_{m-1}(M)$. We define the reverse ordering of eigenvalues $\mu_{1}(M)\geq \mu_{2}(M)\geq \hdots\geq \mu_{m}(M)$, i.e. $\mu_{i}(M)=\lambda_{m-i}(M)$. Therefore we may also define $\mu_{i}(\Sigma)=\mu_{i}(L(\Sigma)).$
The reason we introduce this ordering is that $\lambda_{i}(\Sigma)$ has a close connection to $\mu_{i}(A(\Sigma)).$
\begin{remark}\label{rmk: switching between eigenvalues of graphs}
Let $M$ be a symmetric $m\times m$ matrix. For $i=1,\hdots ,m:$ $$\mu_{i}(-M)=-\mu_{m+1-i}(M).$$
Let $\Sigma$ be an undirected graph. For $i=0,\hdots , \vert V(\Sigma)\vert -1$:
$$\lambda_{i}(\Sigma)=1-\mu_{i+1}(D(\Sigma)^{-1\slash 2}A(\Sigma)D(\Sigma)^{-1\slash 2})=1-\mu_{i+1}(D(\Sigma)^{-1}A(\Sigma)).$$\end{remark}

\begin{remark}
If $\Sigma$ is instead a directed graph, there is an analogous definition of $A(\Sigma)$, $D(\Sigma)$, $L(\Sigma)$, and $\lambda_{i}(\Sigma)$. In particular, $A(\Sigma)_{i,j}$ is equal to the number of directed edges starting at $v_{i}$ and ending at $v_{j}$. Furthermore $D(\Sigma)$ is the diagonal matrix with $D(\Sigma)_{i,i}=deg(v_{i})$ (i.e. the number of directed edges starting at $v_{i}$). 
\end{remark}
We note the following well known lemma that allows us to find an upper bound for $\vert \mu_{i}(A(\Sigma))\vert.$
\begin{lemma}\label{lem: max eigenvalue of adjacency matrix}
If $\Sigma$ is an undirected graph, then 
$$\max_{i}\vert \mu_{i}(A(\Sigma))\vert\leq \max_{v\in V(\Sigma)}deg(v).$$ If $\Sigma$ is bipartite, then 
$$\max_{i}\vert \mu_{i}(A(\Sigma))\vert\leq \max_{\substack{v\in V_{1}(\Sigma)\\w\in V_{2}(\Sigma)}}\sqrt{deg(v)deg(w)}.$$
\end{lemma}
\begin{proof}
The first result follows as $||A(\Sigma)||_{\infty}=\max_{v\in V(\Sigma)}deg(v)$, and it is standard that $||A(\Sigma)||_{\infty}$ is an upper bound for the absolute values of the eigenvalues of $A(\Sigma)$. 

The second inequality follows from e.g. \cite[3.7.2]{hornjohnson}, as we now detail. In this case, we have $$A(\Sigma)=\begin{pmatrix}
0&B\\B^{T}&0
\end{pmatrix},$$ for some matrix $B$. By definition, the set of eigenvalues of $A$ are the set of \emph{singular values} of $B$, $\{\sigma_{j}(B)\}_{j}$. Therefore, $\max_{i}\vert \lambda_{i}(A(G))\vert =\max_{j}\vert \sigma_{j}(B)\vert$. By \cite[3.7.2]{hornjohnson}, $$\max_{j}\vert \sigma_{j}(B)\vert\leq \sqrt{\vert \vert B\vert \vert_{\infty} \vert \vert B\vert \vert_{1}}=\max_{\substack{v\in V_{1}(\Sigma)\\w\in V_{2}(\Sigma)}}\sqrt{deg(v)deg(w)}.$$
\end{proof}

%\begin{proof}
%The first result follows as $||A(G)||_{\infty}=\max_{v\in V(G)}deg(v)$, and it is standard that for any square matrix $M$, $||M||_{\infty}$ is an upper bound for the absolute values of the eigenvalues of $M$. 
%The second inequality follows from e.g. \cite[3.7.2]{hornjohnson}.
%\end{proof}
We now note some matrix inequalities, which we will use to analyse eigenvalues of graphs.
\begin{lemma*}[Weyl's inequality, \cite{Weyl}]
Let $A$ and $B$ be symmetric $m\times m$ real matrices. For $i=1,\hdots ,m$: $\mu_{i}(A)+\mu_{m}(B)\leq \mu_{i}(A+B)\leq \mu_{i}(A)+\mu_{1}(B).$
\end{lemma*}
In a similar spirit to this, we provide an application of the Horn inequalities.
\begin{lemma}\cite[Problem III.6.14.]{BhatiaRajendra1997Ma/R} 
Let $A$ be a $\vert V\vert \times \vert V\vert$ positive real diagonal matrix with minimum entry $A_{min}$ and maximum entry $A_{max}$, and let $B$ be a $\vert V\vert \times \vert V\vert$ symmetric real matrix. Then $$A_{min}\mu_{2}(B)\leq \mu_{2}(AB)\leq A_{max}\mu_{2}(B).$$
\end{lemma}

\begin{definition}
For a vector $\vect{v}=(v_{1},\hdots ,v_{m})$ we define $\norm{v}=\sqrt{\sum v_{i}^{2}}.$ For an $m\times m$ matrix $A$, we define
\begin{enumerate}[label={$\arabic* )$}]
    \item $\vert\vert A\vert\vert_{2}=\max \limits_{\vect{v}\neq \vect{0}}\dfrac{\vert\vert A\vect{v}\vert\vert}{\norm{v}},$
    \item $\vert\vert A\vert\vert_{1}$ as the maximum absolute column sum of $A$, and
    
    \item $\vert\vert A\vert\vert_{\infty }$ as the maximum absolute row sum of $A$.
\end{enumerate}
\end{definition}
We note that $(\vert\vert A\vert\vert_{2})^{2}\leq  \vert\vert A\vert\vert_{1} \vert\vert A\vert\vert_{\infty}.$ Using Weyl's inequality and the fact that $\max_{i}\vert\mu_{i}(A)\vert\leq \vert\vert A\vert \vert_{2}$, we can deduce the following, which will be used implicitly throughout the paper.

\begin{lemma}
Let $A$ and $B$ be symmetric $m\times m$ real matrices. For $i=1,\hdots, m$:
$$\mu_{i}(A)-\vert\vert A-B\vert\vert_{2}\leq \mu_{i}(B)\leq \mu_{i}(A)+\vert\vert A-B\vert\vert_{2}$$
\end{lemma}

Finally, we note the following consequence of the Courant-Fischer theorem: see, e.g. \cite[Corollary III.1.2]{BhatiaRajendra1997Ma/R}.
\begin{theorem*}
Let $A$ be a symmetric $m\times m$ real matrix with first eigenvalue $\mu_{1}(A)$ and corresponding eigenvector $\vect{e}$. Then 
$$\mu_{2} (A)= \max\limits_{\substack{\vect{v}\perp \vect{e}\\ \vert \vert \vect{v}\vert \vert =1}}\langle A\vect{v},\vect{v}\rangle=\max\limits_{\vect{v}\perp \vect{e}}\dfrac{\langle A\vect{v},\vect{v}\rangle}{\langle \vect{v},\vect{v}\rangle}.$$
\end{theorem*}

%----------------------------------------------------------------------------------------
%   Spectral criterion for property (T)
%----------------------------------------------------------------------------------------
%\input{Sections/A_spectral_criterion}
\section{A spectral criterion for Property (T)}\label{sec: a spectral criterion for Property (T)}
In this section we deduce a spectral criterion for Property (T); we first remind the reader of some of the relevant definitions. We focus only on finitely generated groups; for a further exposition the reader should see, for example, \cite{bekkadelaharpe}.

\begin{definition}
Let $\Gamma$ be a finitely generated group with finite generating set $S$, let $\mathcal{H}$ be a Hilbert space, and let $\pi: \Gamma\rightarrow\mathcal{U}(\mathcal{H})$ be a unitary representation of $\Gamma.$ We say that $\pi$ has \emph{almost-invariant vectors} if for every $\epsilon>0$ there is some non-zero $u_{\epsilon}\in\mathcal{H}$ such that for every $s\in S$, $||\pi(s) u_{\epsilon}-u_{\epsilon}||<\epsilon ||u_{\epsilon}||.$ 

We say that $\Gamma$ has \emph{Property (T)} if for every Hilbert space $\mathcal{H}$ and unitary representation $\pi:\Gamma\rightarrow\mathcal{U}(\mathcal{H})$ with almost-invariant vectors, there exists a non-zero invariant vector for $\pi$.
\end{definition}
\begin{remark}
It is standard that the above definition is not reliant on the choice of generating set $S$. Let $\Gamma$ be a finitely generated group. If $H$ is a finite index subgroup of $\Gamma$, then $\Gamma$ has Property (T) if and only if $H$ has Property (T) (see e.g. \cite[Theorem Theorem 1.7.1]{bekkadelaharpe}). If $\Gamma$ has Property (T) and $\Gamma'$ is a homomorphic image of $\Gamma,$ then $\Gamma'$ has Property (T) (see e.g. \cite[Thm 1.3.4]{bekkadelaharpe}).

\end{remark}

Now, let $G=\langle S\;\vert\;R\rangle$ be a finite presentation of a group. We assume that we have fixed a language $\mathcal{L}^{G}\subseteq  (S\sqcup S^{-1})^{*}$ with the following properties:
\begin{enumerate}[label={$\arabic* )$}]
    \item $\mathcal{L}^{G}$ is regular, prefix and suffix closed, and consists of geodesics words, and
    \item for any $l\geq 1$, the set, $\mathcal{L}^{G}_{l}$, of words of length $l$ in $\mathcal{L}^{G}$ generates a finite index subgroup of $G$.
\end{enumerate}
Recall that a language $\mathcal{L}$ is \emph{prefix closed} if for any word $uv\in \mathcal{L}$, $u\in \mathcal{L}$. It is \emph{suffix closed} if for any word $uv\in \mathcal{L}$, $v\in \mathcal{L}$. Finally, it is \emph{regular} if it can be recognised by a finite state automaton. In our application, the language $\mathcal{L}^{G}$ will appear as a set of specific subwords of the language recognised by the canonical automatic structure on a hyperbolic group. 

We define $\mathcal{L}^{G}_{l}$
to be the set of words in $\mathcal{L}^{G}$ of length $l$, and $\mathcal{L}^{G}_{l,\omega}$
to be the set of words in $\mathcal{L}^{G}$ of length between $l-\omega (l)$ and $l$.
Since the evaluation map $\mathcal{E}:\mathcal{L}^{G}\rightarrow G$ is not necessarily injective, we need to make the following definitions.
\begin{definition}
Let $G$ and $\mathcal{L}^{G}$ be as above, let $g\in G$ and $w\in \mathcal{L}^{G}$. We define the following.
\begin{enumerate}[label={$\arabic* )$}]
    \item $\overline{w}$ is the image of $w$ under $\mathcal{E}$.
    \item We say $g\in \mathcal{L}^{G}$ if there exists some $w\in \mathcal{L}^{G}$ with $\overline{w}=_{G}g$.
    \item We define $word(g):=\{w\in \mathcal{L}^{G}\;:\; \overline{w}=_{G}g\}$, and 
    $$\Phi (g):=\vert word (g)\vert.$$
\end{enumerate}
\end{definition}

We now define the graph that plays a central role in our spectral criterion.

\begin{definition}[{\bf The graph} $\boldsymbol{\Upsilon(G,\mathcal{L}^{G},R)}$]
Let $$R=\{r=(x_{r},y_{r},z_{r})\}_{r}$$ be a finite set of triples of elements in $\mathcal{L}^{G}$, such that for each $r\in R$: $x_{r}y_{r}z_{r}$ is reduced without cancellation and lies in $\mathcal{L}^{G}$; and each $x_{r},y_{r},z_{r}$ is reduced and is not the empty word. Let $\mathcal{R}=\{x_{r}y_{r}z_{r}\;:\;r\in R\}$ and $\mathcal{W}(R):=\cup_{r\in R}\{x_{r},y_{r},z_{r}\}.$ Define the directed graph $\Upsilon(G,\mathcal{L}^{G},R)$ as follows. Let $$V(\Upsilon(G,\mathcal{L}^{G},R))=\mathcal{W}(R)\cup(\mathcal{W}(R))^{-1}:=\mathcal{W}\cup \{w^{-1}\;\vert\; w\in \mathcal{W}(R)\}.$$
In particular, we add in all words in $\mathcal{W}(R)$ and their inverses. For each triple $r=(x_{r},y_{r},z_{r})\in R$, add the directed edges $(x_{r},z_{r}^{-1}),(x_{r}^{-1},z_{r}).$ By construction, $A(\Upsilon(G,\mathcal{W},R))$ is symmetric, so that 
$$L(\Upsilon(G,\mathcal{W},R)):=I-D(\Upsilon(G,\mathcal{W},R))^{-1\slash 2}A(\Upsilon(G,\mathcal{W},R))D(\Upsilon(G,\mathcal{W},R))^{-1\slash 2}$$ has real eigenvalues, all of which lie in $[0,2]$. 
For a vertex $w\in V$ (so that $w\neq w^{-1}$) define $$rep (w):=\vert\{(u,w,v)\in  R\} \vert+\vert\{(u,w^{-1},v)\in R\} \vert.$$ For a fixed choice of $R$ and for an element $g\in G$, we define $$deg (g):=\sum\limits_{\substack{w\in V\\\overline{w}=g}}deg_{\Upsilon(G,\mathcal{W},R)}(w),\mbox{ and }rep (g):=\sum\limits_{\substack{w\in V\\\overline{w}=g}}rep(w).$$ 
\end{definition}

We may now introduce our spectral theorem. Recall for $g\in G$, we write $g\in V(\Upsilon(G,\mathcal{L}^{G},R))$ when there exists $w\in V(\Upsilon(G,\mathcal{L}^{G},R))$ with $\overline{w}=_{G}g.$

\begin{theorem}\label{thm: directed new spectral criterion}
Let $G=\langle S\;\vert\; T\rangle$ be a finite presentation of a hyperbolic group. Let $R=\{r=(x_{r},y_{r},z_{r})\}_{r}$ be a finite collection of triples in $\mathcal{L}^{G}$ as above, such that $\mathcal{W}(R)$ generates a finite index subgroup of $G$. Let $\mathcal{R}=\{x_{r}y_{r}z_{r}:r\in R\}$.
Suppose that there exist constants $\epsilon,\;\delta>0$ such that for all $g\in G$ with $g\in V(\Upsilon(G,\mathcal{L}^{G},R))$: $ deg (g)\leq (1+\epsilon) deg(g^{-1})$ and $rep (g)\leq (1+\delta)deg(g). $
If $$\lambda_{1}\bigg(\Upsilon(G,\mathcal{L}^{G},R)\bigg)>\dfrac{1+\delta}{2-4\epsilon^{2}},$$ then $G\slash \langle \langle \mathcal{R}\rangle \rangle$ has Property (T).

\end{theorem}

In order to prove this, we will use a criterion due to Ozawa \cite{ozawapropertyt}. For $\Gamma$ a finitely generated group, let $\mathbb{R}[\Gamma]$ be the real group algebra, i.e. the set of all finitely supported functions $\Gamma\rightarrow \mathbb{R}$. We can write an element $\xi \in\RG{\Gamma}$ as $\sum_{g\in \Gamma}\xi (g)g$. The group algebra comes equipped with an involution map, $*$, given by $\xi^{*}(g)=\xi (g^{-1})$. We define the positive cone $$\Sigma^{2}\RG{\Gamma}=\left\{\sum_{i=1}^{n}\xi_{i}\xi^{*}_{i}\;:\;n\in\mathbb{N},\;\xi_{i}\in \RG{\Gamma}\right\}.$$ For $\xi\in \RG{\Gamma}$, we write $\xi\geq_{\RG{\Gamma}}0$ to indicate that $\xi \in \Sigma^{2}\RG{\Gamma}.$ Similarly, we write $\xi\geq_{\RG{\Gamma}} \phi$ if $\xi -\phi\geq_{\RG{\Gamma}} 0$.

Now, let $A$ be a finite symmetric generating set for $\Gamma$, and let $\nu$ be a symmetric probability measure with support $A$. The Laplacian $\Delta_{\nu}$ is defined as \begin{equation}\label{eq: definition of symmetric group laplacian}\Delta_{\nu}=\frac{1}{2}\sum \limits_{a\in A}\nu(a)(2-a-a^{*})=\frac{1}{2}\sum \limits_{a\in A}\nu(a)(1-a)(1-a)^{*}=1-\sum_{a\in A} \nu(a)a
\end{equation} If $\mu$ is a non-symmetric probability measure with finite support $A$, we define \begin{equation}\label{eq: definition of nonsymmetric group laplacian}
    \Delta_{\mu}:=1-\sum_{a\in A} \mu (a)a.
\end{equation}
We will use the following result of \cite{ozawapropertyt}.
\begin{theorem*}\cite[Main Theorem]{ozawapropertyt}
Let $\Gamma$ be a finitely generated group and $\nu$ a symmetric probability measure on $\Gamma$ whose support generates $\Gamma$. The group $\Gamma$ has Property (T) if and only if there exists a constant $\kappa>0$ such that $\Delta_{\nu}^{2}-\kappa\Delta_{\nu}\geq_{\RG{\Gamma}}0.$
\end{theorem*}
We may now turn to proving Theorem \ref{thm: directed new spectral criterion}. Given an element $\xi\in\RG{\Gamma}$, its $1$-norm is $\vert\vert \xi\vert\vert_{1}=\sum \vert \xi(g)\vert.$ Define $$\I{\Gamma}:=\left\{\xi\in \RG{\Gamma}\;:\;\sum \xi (g)=0\right\},$$ and $$\I{\Gamma}^{h}=\{\xi\in \I{\Gamma}\;:\;\xi^{*}=\xi\}.$$

\begin{proof}[Proof of Theorem \ref{thm: directed new spectral criterion}]
For notational ease, let $\Upsilon=\Upsilon(G,\mathcal{L}^{G},R)$ and $\Gamma=G\slash \langle \langle \mathcal{R}\rangle \rangle$. Suppose $\Upsilon = (V,E)$. The proof of the theorem proceeds in two parts. Firstly, we find: a finite symmetric generating set, $A$, for $\Gamma$; two probability measures, $\mu$, $\overline{\mu},$ on $A$ that will (often) not be symmetric; and a probability measure $\xi$ on $A$ that is symmetric, such that $\Delta_{\mu}\Delta_{\overline{\mu}}-\Delta_{\mu}-\Delta_{\overline{\mu}}+\lambda_{1}(\Upsilon)^{-1}\Delta_{\xi}\geq_{\RG{\Gamma}} 0.$ We then show how to construct a symmetric probability measure $\nu$ such that $\vert \vert \Delta_{\mu}-\Delta_{\nu}\vert \vert_{1}\leq 2\epsilon$. Finally, we show that we can replace $\Delta_{\mu}$, $\Delta_{\overline{\mu}}$, and $\Delta_{\xi}$ by $\Delta_{\nu}$ in the above equation, at the expense of replacing the constant $\lambda_{1}(\Upsilon)^{-1}$ by the constant $(1+\delta)\lambda_{1}(\Upsilon)^{-1}-4\epsilon^{2}.$ The first part of the proof proceeds extremely similarly to \cite[Example 5]{ozawapropertyt}. 

We may assume that $V$ generates $G$. By assumption, $V$ generates a finite index subgroup of $G$ and furthermore, $\mathcal{R}\subseteq \langle V\rangle$. As Property (T) is preserved under finite index extensions and quotients, we may replace $G \slash \langle \langle \mathcal{R}\rangle\rangle  $ by the subgroup $\langle V\rangle \slash \langle \langle \mathcal{R}\rangle\rangle $ if $[G:\langle V\rangle]>1$. 

Let $\phi:G\twoheadrightarrow \Gamma $ be the quotient map, and $A=\{\phi(\overline{w}): w\in V\}$, so that $A$ is a symmetric generating set for $\Gamma$. Let $\mu'$ be the probability measure on $V:=V(\Upsilon)$ given by $\mu' (w)=deg(w)\slash \vert E\vert .$ Let $\mu$ be the probability measure on $A$ given by $$\mu(a):=\sum\limits_{\substack{w\in V\\\phi(\overline{w})=_{\Gamma}a}}deg(w)\slash \vert E\vert=\sum\limits_{\substack{w\in V\\\phi(\overline{w})=_{\Gamma}a}}\mu' (w).$$

Let $\overline{\mu}$ be the probability measure on $A$ defined by $\overline{\mu}(a):=\mu (a^{-1}).$ Next, we define the probability measure $\xi$ on $A$ by 
\begin{align}\label{eq: def of omega}
\xi(a)&:=\sum\limits_{\substack{v\in V\\\phi(\overline{v})=_{\Gamma}a}}\dfrac{rep(v)}{\vert E\vert}\nonumber\\
&=\sum\limits_{\substack{v\in V\\\phi(\overline{v})=_{\Gamma}a}}\dfrac{\vert \{(w^{-1},v)\in E\;:\;(v,u,w)\in R\}\vert+\vert \{(v,w^{-1})\in E\;:\;(v,u^{-1},w)\in R\}\vert}{\vert E\vert }\end{align}

By our assumptions in the statement of the theorem, we have that for all $a\in A$:
\begin{align}\label{eq: omega, mu inequalities}
    \begin{split}
     \xi (a)&\leq (1+\delta)\mu (a),\;\xi (a)\leq (1+\delta)\overline{\mu} (a),\\\;\overline{\mu}(a)&\leq (1+\epsilon)\mu(a),\;\mu (a)\leq (1+\epsilon)\overline{\mu}(a).   
    \end{split} 
\end{align}

Let $\sigma$ be the uniform probability measure on $E$: we define $d:L^{2}(V,\mu')\rightarrow L^{2}(E,\sigma)$ by $(d\zeta)(x,y)=\zeta (y)-\zeta(x)$. Consider the Laplacian $\Lambda:=d^{*}d\slash 2$. As a matrix, $\Lambda=L(\Upsilon)$. In particular, $\Lambda$ has real eigenvalues, all lying in $[0,2]$. Let $\lambda =\lambda_{1}(\Lambda)=\lambda_{1}(\Upsilon)$. Let $P$ be the orthogonal projection from $L^{2}(V,\sigma )$ to constant functions, and $I$ be the identity operator. It is standard (see e.g. \cite[Example 5]{ozawapropertyt}) that $\lambda^{-1}\Lambda+P-I\geq 0 $, so that there is an operator $T$ on $L^{2}(V,\mu')$ with $\lambda^{-1}\Lambda+P-I=T^{*}T$. Given $v,w\in V$ and $O$ an operator on $L^{2}(V,\mu')$, let $O_{v,w}:=\langle O\delta_{v},\delta_{w}\rangle$.

In $\RG{\Gamma}$, we have that:
\begin{align}\label{eq: left hand side of sum is >=0}
   & \sum\limits_{v,w\in V}(\lambda^{-1}\Lambda_{v,w}+P_{v,w}-I_{v,w})\phi(\overline{w})^{-1}\phi(\overline{v})\nonumber\\
   &=_{\RG{\Gamma}} \sum\limits_{v,w\in V}\langle T^{*}T\delta_{v},\delta_{w}\rangle\phi(\overline{w})^{-1}\phi(\overline{v}\nonumber)\\
   & =_{\RG{\Gamma}}\sum\limits_{v,w\in V}\langle T\delta_{v},T\delta_{w}\rangle\phi(\overline{w})^{-1}\phi(\overline{v})\nonumber\\
   (\dagger)\;\;& =_{\RG{\Gamma}}\sum\limits_{v,w\in V}\bigg(\sum_{i\in V}   \mu(i)'^{-1}\langle T\delta_{v},\delta_{i}\rangle\langle \delta_{i},T\delta_{w}\rangle\bigg)\phi(\overline{w})^{-1}\phi(\overline{v})\nonumber\\
   &=_{\RG{\Gamma}}\sum\limits_{v,w\in V}\sum\limits_{i\in V}\mu(i)^{-1}T_{i,g}T_{i,h}\phi(\overline{w})^{-1}\phi(\overline{v})\nonumber\\
    &=_{\RG{\Gamma}}\sum\limits_{i\in V }\zeta_{i}^{*}\zeta_{i}\geq_{\RG{\Gamma}}0,\end{align}
where we have defined $$\zeta_{i}=\mu'(i)^{-1\slash 2}\sum \limits_{v\in V}T_{i,v}\phi(\overline{v})\in \RG{\Gamma}.$$ The inequality $(\dagger)$ follows by definition of the inner product on $L^{2}(V,\mu')$.

We calculate that:
\begin{align*}
    \Lambda_{v,v}&=\dfrac{1}{2\vert E\vert}\sum_{(x,y)\in E}(\delta_{v}(x)-\delta_{v}(y))^{2}\\
    &=\sum_{y\neq v,(v,y)\in E}\dfrac{1}{2\vert E\vert }+\sum_{y\neq v,(y,v)\in E}\dfrac{1}{2\vert E\vert }\\
    &=\dfrac{\left\vert \{(v,y)\in E\} \right\vert-\left\vert \{(v,v)\in E\} \right\vert}{\vert E\vert }\\
    &=\mu'(v)-\dfrac{\left\vert \{(v,v)\in E\} \right\vert}{\vert E\vert }
\end{align*}

For $v\neq w$, $\Lambda_{v,w}=-\vert E\vert ^{-1}$ if $(v,w)\in E$, and $\Lambda_{v,w}=0$ otherwise.  Hence:
\begin{align}\label{eq: Relating sum Lvw to Delta omega}
   &\sum\limits_{v,w\in V}\lambda^{-1}\Lambda_{v,w}\phi(\overline{w})^{-1}\phi(\overline{v}\nonumber)\\
   &=_{\RG{\Gamma}}\sum \limits_{v\in V}\mu'(v)\phi(\overline{v})^{-1}\phi(\overline{v})-\sum \limits_{(v,w)\in E}\dfrac{\phi(\overline{w})^{-1}\phi(\overline{v})}{\vert E\vert }\nonumber\\
      &=_{\RG{\Gamma}}\sum \limits_{v\in V}\mu'(v)-\sum \limits_{(v,w)\in E}\dfrac{\phi(\overline{w})^{-1}\phi(\overline{v})}{\vert E\vert }\nonumber\\
   (\ddag)& =_{\RG{\Gamma}}1-\dfrac{1}{\vert E\vert }\sum_{a\in A}\sum\limits_{\substack{u\in V\\\phi(\overline{u})=_{\Gamma}a}}\vert\{(y^{-1},x)\in E\;:\;(x,u,y)\in R\}\vert a\nonumber\\
   &\hspace*{35 pt}-\dfrac{1}{\vert E\vert }\sum_{a\in A}\sum\limits_{\substack{u\in V\\\phi(\overline{u})=_{\Gamma}a}}\vert \{(x,y^{-1})\in E\;:\;(x,u^{-1},y)\in R\}\vert a\nonumber\\
   &=_{\RG{\Gamma}}\Delta_{\xi}\;\;(\mbox{by equations }(\ref{eq: definition of nonsymmetric group laplacian})\mbox{ and }(\ref{eq: def of omega})).
\end{align}
The equation $(\ddag)$ follows since $(x,u,y)\in R$ implies that $xuy=_{\Gamma}1$, so that $a=\phi(u)=\phi(x)^{-1}\phi(y)^{-1}$, and hence $a$ arises from the edge $(y^{-1},x).$ 
Similarly, we have that $P_{v,w}=\mu'(v)\mu'(w),$ and $ I_{v,w}=\delta_{v,w}\mu'(v).$ Therefore:
\begin{align}\label{eq: calculation LHS of P and I}
    &\sum\limits_{v,w\in V}(P_{v,w}-I_{v,w})\phi(\overline{w})^{-1}\phi(\overline{v})\nonumber\\
    &=_{\RG{\Gamma}}\sum\limits_{v,w\in V}\mu'(v)\mu'(w)\phi(\overline{w})^{-1}\phi(\overline{v})-\sum\limits_{v\in V}\mu'(v)\phi(\overline{v})^{-1}\phi(\overline{v})\nonumber\\
     &=_{\RG{\Gamma}}\sum\limits_{v,w\in V}\mu'(v)\mu'(w)\phi(\overline{w})^{-1}\phi(\overline{v})-\sum\limits_{v\in V}\mu'(v)\nonumber\\
&=_{\RG{\Gamma}}\bigg(\sum\limits_{v\in V}\mu'(v)\phi(\overline{v})\bigg)\bigg(\sum\limits_{v\in V}\mu'(v^{-1})\phi(\overline{v})\bigg)-1\nonumber\\
&=_{\RG{\Gamma}}(1-\Delta_{\mu})(1-\Delta_{\overline{\mu}})-1\nonumber\\
&=_{\RG{\Gamma}}\Delta_{\mu}\Delta_{\overline{\mu}}-\Delta_{\mu}-\Delta_{\overline{\mu}}.
\end{align}
Therefore, by equations $(\ref{eq: left hand side of sum is >=0})$, $(\ref{eq: Relating sum Lvw to Delta omega})$, and $(\ref{eq: calculation LHS of P and I})$:
$$\Delta_{\mu}\Delta_{\overline{\mu}}-\Delta_{\mu}-\Delta_{\overline{\mu}}+\lambda^{-1}\Delta_{\xi}\geq_{\RG{\Gamma}}0.$$
Note that $\Delta_{\overline{\mu}}:=_{\RG{\Gamma}}\Delta_{\mu}^{*}$. Since $\xi$ is symmetric, by taking the involution, we see that
$$\Delta_{\overline{\mu}}\Delta_{\mu}-\Delta_{\mu}-\Delta_{\overline{\mu}}+\lambda^{-1}\Delta_{\xi}\geq_{\RG{\Gamma}}0.$$ 
Therefore: 
\begin{equation}\label{eq: first symmetric inequality}
    \dfrac{1}{2}\bigg(\Delta_{\overline{\mu}}\Delta_{\mu}+\Delta_{\mu}\Delta_{\overline{\mu}}\bigg)-\Delta_{\mu}-\Delta_{\overline{\mu}}+\lambda^{-1}\Delta_{\xi}\geq_{\RG{\Gamma}}0.
\end{equation}

Let $\nu:=(\mu+\overline{\mu})\slash 2$, so that $\nu$ is a symmetric probability measure on $A$. We see that $\Delta_{\nu}=_{\RG{\Gamma}}(\Delta_{\mu}+\Delta_{\overline{\mu}})\slash 2$. Furthermore: 
\begin{align}\label{eq: delta nu squared}
    \Delta_{\nu}^{2}&=_{\RG{\Gamma}}\dfrac{1}{4}(\Delta_{\mu}+\Delta_{\overline{\mu}})^{2}=_{\RG{\Gamma}}\dfrac{1}{4}\bigg(\Delta_{\mu}^{2}+\Delta_{\overline{\mu}}^{2}+\Delta_{\mu}\Delta_{\overline{\mu}}+\Delta_{\overline{\mu}}\Delta_{\mu}\bigg)\nonumber\\
    &=_{\RG{\Gamma}}\dfrac{1}{2}\bigg(\Delta_{\mu}\Delta_{\overline{\mu}}+\Delta_{\overline{\mu}}\Delta_{\mu}\bigg)+\dfrac{1}{4}\bigg(\Delta_{\mu}^{2}+  \Delta_{\overline{\mu}}^{2}-\Delta_{\mu}\Delta_{\overline{\mu}}-\Delta_{\overline{\mu}}\Delta_{\mu}\bigg)\nonumber\\
    &=_{\RG{\Gamma}}\dfrac{1}{2}\bigg(\Delta_{\mu}\Delta_{\overline{\mu}}+\Delta_{\overline{\mu}}\Delta_{\mu}\bigg)+\dfrac{1}{4}\bigg(\Delta_{\mu}-\Delta_{\overline{\mu}}\bigg)\bigg(\Delta_{\mu}-\Delta_{\overline{\mu}}\bigg).
\end{align}

It is easy to see that $(1+\delta)\Delta_{\nu}\geq_{\RG{\Gamma}} \Delta_{\xi},$ since $(1+\delta)\mu(a)\geq \xi(a)$ for all $a\in A$, and similarly for $\overline{\mu}(a)$. Let $\Sigma=(\Delta_{\mu}-\Delta_{\overline{\mu}})(\Delta_{\mu}-\Delta_{\overline{\mu}})\in \I{\Gamma}.$ By equations $(\ref{eq: first symmetric inequality})$ and $(\ref{eq: delta nu squared})$, we have $$\Delta_{\nu}^{2}-2\Delta_{\nu}+\lambda^{-1}\Delta_{\xi}\geq_{\RG{\Gamma}} -\Sigma\slash 4,$$ and hence \begin{equation}\label{eq: second symmetric inequality}
    \Delta_{\nu}^{2}-(2-(1+\delta)\lambda^{-1})\Delta_{\nu}\geq_{\RG{\Gamma}} -\Sigma\slash 4.
\end{equation} To apply Ozawa's criterion, it therefore remains to analyze $\Sigma.$ We note that, by e.g. \cite[Lemma 2]{ozawapropertyt}, for any $a,b\in \Gamma$: 
\begin{equation}\label{eq: Ozawas lemma}
    (2-ab-(ab)^{*})\leq_{\RG{\Gamma}} 2(2-a-a^{*})+2(2-b-b^{*}).
\end{equation} Now, $\Sigma^{*}=\Sigma$ (in particular $\Sigma(g)=\Sigma(g^{-1})$), and $\sum \Sigma (g)=0,$ so that $\Sigma\in\I{\Gamma}^{h}$. We may therefore write $$\Sigma = \dfrac{1}{2}\sum \limits_{a,b\in A}\Sigma (ab) (2-(ab)^*-ab).$$ This is a standard characterization of such elements, as the cone $\I{\Gamma}^{h}$ is spanned by the elements $2-x-x^{-1}$: see e.g. \cite{ozawapropertyt}. By comparing coefficients with the definition of $\Sigma$, we see that $$\Sigma(ab)=(\mu(a)-\mu(a^{-1}))(\mu(b)-\mu(b^{-1})).$$
Therefore, as $\mu(a)\leq (1+\epsilon)\mu(a^{-1})$ for all $a\in A$: 

\begin{align}
    \Sigma&=_{\RG{\Gamma}}\dfrac{1}{2}\sum\limits_{a,b\in A}(\mu(a)-\mu(a^{-1}))(\mu(b)-\mu(b^{-1}))(2-ab-(ab)^{*})\nonumber\\
    (\ddag \dagger)\;&\leq _{\RG{\Gamma}}\dfrac{1}{2}\sum\limits_{a,b\in A}2(\mu(a)-\mu(a^{-1}))(\mu(b)-\mu(b^{-1}))(2-a-a^{*})\nonumber\\
    &\;\;\;\;+\dfrac{1}{2}\sum\limits_{a,b\in A}2(\mu(a)-\mu(a^{-1}))(\mu(b)-\mu(b^{-1}))(2-b-b^*)\nonumber\\
    &=_{\RG{\Gamma}} 2\sum_{b\in A}(\mu(b)-\mu(b^{-1}))\sum_{a\in A}(\mu(a)-\mu(a^{-1}))(2-a-a^{*})\nonumber\\
    &\leq_{\RG{\Gamma}}2\sum_{b\in A}2\epsilon \min\{\mu (b),\mu (b^{-1})\} \sum_{a\in A}(\mu(a)-\mu(a^{-1}))(2-a-a^{*})\nonumber\\
     &\leq_{\RG{\Gamma}}4\epsilon \sum_{a\in A}(\mu(a)-\mu(a^{-1}))(2-a-a^{*})\nonumber\\
    (\ddag \ddag)&\leq_{\RG{\Gamma}}4\epsilon\sum_{a\in A}2\epsilon \min\{\mu (a),\mu (a^{-1})\}(2-a-a^{*})\nonumber\\
   (\ddag \ddag \dagger)&\leq_{\RG{\Gamma}}8\epsilon^{2}\sum_{a\in A}\nu(a)(2-a-a^{*}) \nonumber\\
    &=_{\RG{\Gamma}}16\epsilon^{2}\Delta_{\nu}.
\end{align}

We have the (in)equality $(\ddag \dagger)$ by equation $(\ref{eq: Ozawas lemma})$, $(\ddag \ddag)$ by equation $(\ref{eq: omega, mu inequalities})$, and $(\ddag \ddag\dagger )$ by the definition of $\nu$.
Dividing by $-1\slash 4$, this implies that 
\begin{equation}\label{eq: Sigma inequality}
    -\Sigma\slash 4\geq -4\epsilon^{2}\Delta_{\nu}.
\end{equation}
Hence, by equations $(\ref{eq: second symmetric inequality})$ and $(\ref{eq: Sigma inequality})$, we have that:
\begin{align*}
\Delta_{\nu}^{2}-(2-(1+\delta)\lambda^{-1})\Delta_{\nu} \geq_{\RG{\Gamma}}-\Sigma\slash 4 \geq_{\RG{\Gamma}}-4\epsilon^{2}\Delta_{\nu},
\end{align*}
so that $$\Delta_{\nu}^{2}-(2-4\epsilon^{2}-(1+\delta)\lambda^{-1})\Delta_{\nu}\geq_{\RG{\Gamma}}0.$$ Set $\kappa = 2-4\epsilon^{2}-(1+\delta)\lambda^{-1};$ since $\lambda>(1+\delta)\slash(2-4\epsilon^{2}),$ we have that $\kappa >0$. Therefore, $\nu$ is a symmetric probability measure on $\Gamma$, with finite support that generates $\Gamma$, and there exists $\kappa>0$ such that $$\Delta_{\nu}^{2}-\kappa \Delta_{\nu} \geq_{\RG{\Gamma}}0.$$ Hence, by \cite{ozawapropertyt}, $\Gamma=G\slash \langle \langle R\rangle \rangle$ has Property (T).
\end{proof}

This allows us to deduce the following.

\begin{cor}\label{cor: directed new spectral criterion}
Let $G=\langle S\;\vert\; T\rangle$ be a hyperbolic group. For each $l\geq 1$, choose a set of triples $R_{l}=\{r=(x_{r},y_{r},z_{r})\}_{r}$ in $\mathcal{L}^{G}$ such that: each $x_{r}$, $y_{r}$, $z_{r}$ has length between least $l\slash 3-\omega (l)\slash 3$ and $l\slash 3+2$; and $x_{r}y_{r}z_{r}$ is reduced without cancellation and lies in $\mathcal{L}^{G}_{l,\omega}$. Let $\mathcal{R}_{l}=\{x_{r}y_{r}z_{r}\;:\;r\in R\},$ and $G\slash \langle \langle \mathcal{R}_{l}\rangle \rangle $ be a series of quotients of $G$. Suppose that $\mathcal{L}_{\lfloor l\slash 3\rfloor}^{G}\subseteq V(\Upsilon(G,\mathcal{L}^{G},R_{l}))$, and for all $g \in V(\Upsilon(G,\mathcal{L}^{G},R_{l}))$: $$ deg (g)=(1+o_{l}(1))deg(g^{-1})= (1+o_{l}(1)) rep (g), $$
where the $o_{l}(1)$ term is independent of $g$. If $\lim_{l\rightarrow\infty}\lambda_{1}(\Upsilon(G,\mathcal{L}^{G},R_{l}))>1\slash 2$, then $G\slash \langle \langle \mathcal{R}_{l}\rangle \rangle$ has Property (T) for all sufficiently large $l$.
\end{cor}
\begin{proof}
Since $\lambda=\lim_{l\rightarrow \infty}\lambda_{1}(\Upsilon(G,\mathcal{L}^{G},R_{l}))>1\slash 2,$ we may choose $\epsilon, \delta >0$ such that 
$$\lambda>(1+\delta)\slash (2-4\epsilon^{2}).$$ Then $$\lambda_{1}(\Upsilon(G,\mathcal{L}^{G},R_{l}))>(1+\delta)\slash (2-4\epsilon^{2})$$ for all sufficiently large $l$. Similarly, by the assumptions of the corollary, we may assume that for all sufficiently large $l$ and for all $g\in V(\Upsilon(G,\mathcal{L}^{G},R_{l}))$: $ deg (g)\leq(1+\epsilon)deg(g^{-1})$ and $rep (g)\leq(1+\delta)deg (g).$ Since $\mathcal{L}^{G}_{\lfloor l\slash 3\rfloor}\subseteq V(\Upsilon(G,\mathcal{L}^{G},R_{l}))$, we see that $\mathcal{W}(R)$ generates a finite index subgroup of $G$. Applying Theorem \ref{thm: directed new spectral criterion} allows us to conclude the result.
\end{proof}

%----------------------------------------------------------------------------------------
%   Counting geodesics in hyperbolic groups
%----------------------------------------------------------------------------------------
%\input{Sections/Counting_geodesics_in_hyperbolic_groups}

\section{Counting geodesics in hyperbolic groups}\label{sec: counting geodesics in hyperbolic groups}
Now that we have a spectral criterion for Property (T), we have two remaining tasks. Firstly, we need to find our language $\mathcal{L}^{G}$, and secondly, we have to analyse the resulting graphs $\Upsilon$. We turn to the theory of geodesics in hyperbolic groups. We wish to encode the geodesics in a hyperbolic group as an easily understood dynamical system. The most natural way to do this is to use \emph{Markov chains} and \emph{subshifts}. The reader should see, for example, \cite{calegari2013ergodic} for an in-depth discussion of these concepts. Let $G=\langle S\;\vert\;T\rangle$ be a finite presentation of a non-elementary hyperbolic group (we may assume that $S$ is inverse-closed). Assign an ordering to $S$, and extend this to the \emph{lexicographical ordering} of words in $S^{*}$. 

\begin{definition}
A \emph{finite state automaton} (\emph{FSA}) is a finite directed graph with a single distinguished start vertex $v_{0}$, with edges are labelled by elements of $S$ such that for each vertex $v$ and each $s\in S$ there is at most one edge incident to $v$ bearing the label $s$. Some of the vertices are marked as accept vertices. The \emph{language parameterized} by the FSA is the set of words that can be read by starting at $v_{0}$ and concatenating edge labels along a path that ends in an accept state.
\end{definition}
The automatic structure of hyperbolic groups was analysed in \cite{Cannon_combinatorial_structure_hyperbolic_groups}. We will make great use of the following.

\begin{definition}
A subset $\mathcal{L}\subseteq S^{*}$ is a \emph{combing for} $G$ if: $\mathcal{L}$ bijects with $G$ under evaluation; $\mathcal{L}$ is regular (i.e. there exists an FSA parameterizing $\mathcal{L}$); $\mathcal{L}$ is prefix closed, i.e. if $vw\in \mathcal{L}$, then $w\in \mathcal{L}$; and each word in $\mathcal{L}$ is geodesic in $G$.
\end{definition}
Given a language $\mathcal{L}$ and $l\geq 1$, $\mathcal{L}_{l}$ is the subset of $\mathcal{L}$ of words of length $l$. We define $\mathcal{L}_{\infty}$ to be the set of right-infinite words of $\mathcal{L}$. The following is one of the major results in the study of hyperbolic groups.

\begin{theorem*}\cite{Cannon_combinatorial_structure_hyperbolic_groups}
Let $G=\langle S\;\vert\;T\rangle$ be a finite presentation of a hyperbolic group, and let $\mathcal{L}^{geo}$ be the language of lexicographically first geodesics. Then $\mathcal{L}^{geo}$ is a combing for $G$.
\end{theorem*}

In particular, there is a directed graph $\Sigma = (V,E)$ parameterizing $\mathcal{L}^{geo}$ with a unique start vertex $v_{0}$ and all other vertices accept vertices \cite{Cannon_combinatorial_structure_hyperbolic_groups} (see e.g. \cite{calegari2013ergodic} for an overview of its construction). For $n\geq 1$, let $Y_{n}$ be the set of paths of length $n$ in $\Sigma$ starting at $v_{0}$. Define $X_{n}$ as the set of all paths of length $n$ in $\Sigma$. Note that  $Y_{n}\subseteq X_{n}$. Let $\mathcal{E}:\sqcup_{n}X_{n}\rightarrow G$ be the evaluation map, i.e. map a path to the corresponding path in $G$ starting at $e$ and take its endpoint in $G$. We can prove the following.
\begin{lemma}\label{lem: evaluation map is bounded to one}
The evaluation map $\mathcal{E}:\sqcup_{n}X_{n}\rightarrow G$ is bounded-to-one, with bound $\vert V(\Sigma)\vert.$ 
\end{lemma}
\begin{proof}
Suppose that $\sigma, \sigma ' $ are two paths in $\Sigma$ starting from the same vertex, $v$, and evaluating to the same element in $G$. Let $\gamma$ be a path from the start vertex $v_{0}$ to $v$. Then $\gamma \sigma$ and $\gamma \sigma'$ are two paths in $Y_{\vert \gamma\vert +\vert \sigma\vert}$ evaluating to the same element. Since $\mathcal{L}^{geo}$ is a combing, we have that $\gamma\sigma=\gamma \sigma'$ as words, and so $\sigma =\sigma'$ as words. Therefore, for any $g\in G$ and any vertex $v$ of $\Sigma$, there is at most one path in $\Sigma$ starting at $v$ and evaluating to $g$. This implies that the evaluation map has preimage size bounded by $\vert V(\Sigma ) \vert$. 
\end{proof}
A \emph{component} $C$ of $\Sigma$ is a maximal subgraph of $\Sigma$ such that there is a directed path from any vertex $v$ in $C$ to any vertex $w$ in $C$. In this case, by an abuse of notation we can view a component $C$ as a FSA, $\Gamma_{C}$, by adding a unique start vertex $v_{0}'$ adjacent to every vertex in $C$ with each of these added edges having label the empty word, and make all other vertices accept states. Therefore, we can write $\mathcal{L}(C)=\mathcal{L}(\Gamma_{C}).$ In fact, $C$ is an \emph{irreducible (topological) Markov chain} (see e.g. \cite[p. 23]{calegari2013ergodic}).

For such a component, let $M(C)$ be its adjacency matrix, and let $M$ be the adjacency matrix of $\Sigma$. Let $\mu$ be the maximal real eigenvalue of $M$ and $\mu(C)$ the maximal real eigenvalue of $M(C)$. A component $C$ is \emph{maximal} if $\mu(C)=\mu$. We remark that in actuality $\mu = \exp\{\mathfrak{h}\}$ and the number of paths of length $n$ in $M(C)$ is $\eta (C)(1+o(1))\mu(C)^{n}$ for some constant $\eta(C)>0$.

Since each component $C$ is connected, by the Perron--Frobenius theorem, $M_{C}$ has a unique real eigenvalue $\mu (C)$ of maximal norm (with algebraic multiplicity $1$).

The following Lemma appears as a component of the proof of Proposition 6.2 in \cite{gekhtman2018counting}.

\begin{lemma}\label{lem: Cmax gens finite index}\cite[Proof of Proposition 6.2]{gekhtman2018counting}
There exists a maximal component, $C^{max}$, of $\Sigma$ such that for any $l\geq 1$, the set $\mathcal{L}_{l}(C^{max})$ generates a finite-index subgroup of $G$.
\end{lemma}
We may now fix $C^{max}$ to be the desired maximal component in the above lemma and let $\mathcal{L}^{G}=\mathcal{L}(C^{max})$ be the language of $C^{max}.$

We wish to encode the geodesics in a hyperbolic group as an easily understood dynamical system. The most natural way to do this is to use \emph{shifts}.

\begin{definition}
Let $V$ be a finite set. The \emph{two-sided shift on $V$} is the topological space $X:=V^{\mathbb{Z}}$. We endow $V$ with the discrete topology and $X$ with the product topology. It comes endowed with the \emph{shift map} $\sigma:X\rightarrow X$ defined by $(\sigma(\vect{x}))_{i}=\vect{x}_{i+1}.$
\end{definition}
A \emph{subshift} is a closed $\sigma$-invariant subspace of $X$. We wish to look at a specific type of type of shift, called a \emph{subshift of finite type}.
\begin{definition}
Let $A$ be a $\vert V\vert\times \vert V\vert$ $0-1$ matrix. The \emph{subshift of finite type associated to $A$} is the subshift $X_{A}\subset X$ defined by 
$$X_{A}=\left\{\vect{x}\in X\;:\;A_{\vect{x}_{i},\vect{x}_{i+1}}=1\right\}.$$
\end{definition}
If we take $\Sigma = (V,E)$ to be a finite directed graph, the \emph{edge subshift} associated to $\Sigma$ is the set $X_{\Sigma}\subseteq (E)^{\mathbb{Z}}$ that is defined by the incidence matrix of $\Sigma$ (recall the incidence matrix $B$ is the $\vert E\vert\times \vert E\vert $ $0-1$ matrix, where $B_{i,j}=1$ if the endpoint of $e_{i}$ is equal to the startpoint of $e_{j}$, and is zero otherwise).

Let $\mathcal{Z}$ be the set of bi-infinite paths in $C^{max},$ i.e. it is the edge subshift of $C^{max}$. Since $C^{max}$ is connected, we see that $B$ (the incidence matrice) is irreducible, and furthermore, by the Perron-Frobenius theorem, it has unique largest real eigenvalue; using the growth rate of $\mathcal{Z}$, this eigenvalue is exactly $\mu$.

\begin{definition}[The Parry measure]
Let $\alpha,\beta$ be the left and right $\mu$-eigenvectors for $B$, normalised so that $\alpha\beta =1.$ Let $p_{i}=u_{i}v_{i}$, and $P_{i,j}=B_{i,j}\beta_{j}\slash \mu \beta_{i}$.

For an edge path $\gamma=(e_{i_{0}},e_{i_{1}},\hdots ,e_{i_{n}})$ in $C^{max}$ and $n\geq 0$, the \emph{cylinder set} $\mathcal{Z}[\gamma,m]$ is defined as

$$\mathcal{Z}[\gamma,m]=\left \{\underline{z}\in\mathcal{Z}\;:\;(z_{m},z_{m+1},\hdots ,z_{m+n})=\gamma\right\}.$$
These sets form a clopen basis for the subshift. The \emph{Parry measure} on $\mathcal{Z}$, $\nu$, defined in \cite{parry1964intrinsic}, is given by its definition on cyclinder sets
$$\nu( \mathcal{Z}[(e_{i_{0}},e_{i_{1}},\hdots ,e_{i_{n}}),m])=p_{i_{0}}P_{i_{0},i_{1}}\hdots P_{i_{n-1},i_{n}}=\dfrac{\alpha_{i_{0}}\beta_{i_{n}}}{\mu^{n}}.$$

\end{definition}

We now begin to analyse the number of paths in $C^{max}$, weighted by the Parry measure. First, we prove the following.

\begin{lemma}\label{lem: convergence of approximation}
For $\omega(L)$ a small-growing function, define $$B_{L,\omega}=\dfrac{1}{\omega(L)}\sum_{m=L-\omega}^{L}\dfrac{1}{\mu^{m}}B^{m}.$$ Then $B_{L,\omega}\rightarrow \beta\alpha$.
\end{lemma}
\begin{proof}
We calculate that:
\begin{align*}
   \left \vert\left\vert B_{L+1,\omega}-B_{L,\omega}\right\vert\right\vert
&\leq \dfrac{1}{\omega (L+1)\mu^{L+1}}\left\vert \left\vert B^{L+1}\right\vert\right\vert+\dfrac{1}{\omega (L)}\sum_{m=L-\omega (L)}^{L+1-\omega (L+1)}\dfrac{1}{\mu^{m}}\vert\vert B^{m}\vert\vert\\
&+\sum_{m=\max\left\{\substack{L+1-\omega (L+1)\\L-\omega (L)}\right\}}^{L}\left(\dfrac{1}{\omega(L)}-\dfrac{1}{\omega(L+1)}\right)\dfrac{1}{\mu^{m}}\vert\vert B^{m}\vert\vert\\
&\leq \dfrac{1}{\omega(L+1)}+\dfrac{\omega(L+1)+1-\omega(L)}{\omega (L)}+\sum\limits_{m=\min\left\{\substack{L+1-\omega (L+1)\\L-\omega (L)}\right\}}^{L}
\dfrac{\omega (L+1)-\omega (L)}{\omega(L)\omega(L+1)}\\
&\leq o_{L}(1)+ 2\dfrac{\omega(L+1)-\omega (L)}{\omega (L)}\\
&=o_{L}(1).
\end{align*}
Hence $B_{L,\omega}$ converges to a matrix $B_{\omega}$. Next,
\begin{align*}
   \left \vert\left\vert \dfrac{1}{\mu}BB_{L,\omega}-B_{L+1,\omega}\right\vert\right\vert
&\leq \sum_{m=L-\omega (L)}^{L+1-\omega (L+1)}\dfrac{1}{\omega(L)}+ \sum_{m=L+1-\omega (L+1)}^{L+1}\left(\dfrac{1}{\omega(L)}-\dfrac{1}{\omega(L+1)}\right)\dfrac{1}{\mu^{m}}\vert\vert B^{m}\vert\vert\\
&\leq o_{L}(1).
\end{align*}
Therefore $$\dfrac{1}{\mu}BB_{\omega}=B_{\omega}=\dfrac{1}{\mu}B_{\omega}B.$$

Since $B$ has a unique real eigenvalue $\mu$ of largest norm, by applying the above we see that $B_{\omega}$ has eigenvalues $1$ (of algebraic multiplicity $1$) and $0$. In particular, $B_{\omega}$ is a rank one projection, and so can be written $$B_{\omega}=\beta\alpha,$$
where $\alpha$ (respectively $\beta$) is the left (respectively right) $\mu$-eigenvector of $B_{\omega}$, and hence of $B$. 

\end{proof}
We now observe the following.
\begin{lemma}\label{lem: measure of sets between two paths}
Let $\gamma=(e_{i_{1}},\hdots ,e_{i_{m}}), \gamma'=(e_{j_{1}},\hdots ,e_{j_{n}})$ be paths in $C^{max}$. Then $$\sum \limits_{\substack{L-\omega (L)\leq s\leq L\\\sigma=(e_{k_{1}},\hdots e_{k_{s}})}}\nu (\mathcal{Z}[\gamma,\sigma,\gamma'],0)=(1+o_{L}(1))\omega (L) \alpha_{i_{1}}\beta_{i_{m}}\alpha_{j_{1}}\beta_{j_{n}}\slash \mu^{m+n-1}.$$
\end{lemma}

\begin{proof}
We see that 
\begin{align*}
    \sum \limits_{\substack{L-\omega (L)\leq s\leq L\\\sigma=(e_{k_{1}},\hdots e_{k_{s}})\\\gamma\sigma\gamma'\mbox{ a path in  }C^{max}}}\nu (\mathcal{Z}[\gamma,\sigma,\gamma'],0)&=\sum \limits_{\substack{L-\omega (L)\leq s\leq L\\\sigma=(e_{k_{1}},\hdots e_{k_{s}})\\\gamma\sigma\gamma'\mbox{ a path in  }C^{max}}}\alpha_{i_{1}}\beta_{j_{n}}\slash\mu^{m+n+s-1}\\
 &=   \dfrac{\alpha_{i_{1}}\beta_{j_{n}}}{\mu^{m+n-1}}\sum \limits_{\substack{L-\omega (L)\leq s\leq L\\\sigma=(e_{k_{1}},\hdots e_{k_{s}})\\\gamma\sigma\gamma'\mbox{ a path in  }C^{max}}}1\slash\mu^{s}\\
  &=\dfrac{\alpha_{i_{1}}\beta_{j_{n}}}{\mu^{m+n-1}}\sum \limits_{L-\omega (L)\leq t\leq L}B^{t}_{i_{0},j_{n}}1\slash\mu^{s}\\
    &=(1+o_{L}(1))\omega (L) \alpha_{i_{1}}\beta_{i_{m}}\alpha_{j_{1}}\beta_{j_{n}}\slash \mu^{m+n-1}
\end{align*}
by Lemma \ref{lem: convergence of approximation}.
\end{proof}
We can also immediately deduce the following lemmas.
\begin{lemma}\label{lem: measure of starting path}
Let $\gamma=(e_{i_{1}},\hdots ,e_{i_{m}}),$ be a path in $C^{max}$. Then 
\begin{align*}
 \sum \limits_{\substack{L-\omega (L)\leq n\leq L\\\sigma=(e_{j_{1}},\hdots ,e_{j_{n}})}}&\sum \limits_{\substack{L-\omega (L)\leq s\leq L\\\sigma'=(e_{k_{1}},\hdots e_{k_{s}})}}\nu (\mathcal{Z}[\gamma,\sigma,\sigma'],0)=\omega (L)^{2} \nu (\mathcal{Z}[\gamma])\\
 &=\sum \limits_{\substack{L-\omega (L)\leq n\leq L\\\sigma=(e_{j_{1}},\hdots ,e_{j_{n}})}}\sum \limits_{\substack{L-\omega (L)\leq s\leq L\\\sigma'=(e_{k_{1}},\hdots e_{k_{s}})}}\nu (\mathcal{Z}[\sigma,\sigma'\gamma],0).   
\end{align*}
\end{lemma}

\begin{lemma}\label{lem: measure of sets middle path}
Let $\gamma=(e_{i_{1}},\hdots ,e_{i_{m}})$ be a path in $C^{max}$. Then $$\sum \limits_{\substack{L-\omega (L)\leq n\leq L\\\sigma=(e_{j_{1}},\hdots ,e_{j_{n}})}}\sum \limits_{\substack{L-\omega (L)\leq l\leq s\\\sigma'=(e_{k_{1}},\hdots e_{k_{s}})}}\nu (\mathcal{Z}[\sigma,\gamma,\sigma'],0)= \omega(L)^{2}\nu(\mathcal{Z}[\gamma,0]).$$
\end{lemma}
%----------------------------------------------------------------------------------------
%   First eigenvalue of some random graphs
%----------------------------------------------------------------------------------------
%\input{Sections/First_eigenvalue_of_random_graphs_including_2_torsion}
\section{The first eigenvalue of some random graphs}\label{appendix: Spectral theory of restricted graphs}
To prove Theorem \ref{mainthm: property t in random quotients of hyperbolic groups}, we are required to analyse the eigenvalues of a particular model of random graphs. We first introduce the following well known model.
\begin{definition}
Let $m\geq 1$ and let $\underline{w}=(w_{1},\hdots, w_{m})\in\mathbb{Z}_{+}^{m}.$
The random graph $\Gw$ is the graph with vertex set $V=\{u_{1},\hdots,u_{m}\}$, and each edge between $u_{i}$ and $u_{j}$ is added with probability $w_{i}w_{j}\rho$, where $\rho=1\slash\sum w_{i}.$ We define $w_{min}=\min\{w_{i}\}$ and $\overline{w}=\sum w_{i}\slash m.$

\end{definition}
The eigenvalues of this graph were analysed in \cite{chungrandomgraph}.
\begin{theorem*}\cite[Theorem 5]{chungrandomgraph}
Suppose that $m$ and $\underline{w}$ are defined such that \\$w_{min}=\Omega_{m} (\sqrt{\;\overline{w}}\log^{3}m)$. Let $\Sigma\sim \Gw$. Then a.a.s.$(m)$, $$\max_{i\neq 0}\vert 1-\lambda_{i}(\Sigma) \vert= 2[1+o_{m}(1)]\slash \sqrt{\;\overline{w}}.$$
\end{theorem*}
The above inequality can be rewritten as $$\max_{i\neq 1}\vert \mu_{i}(D(\Sigma)^{-1}A(\Sigma)) \vert= 2[1+o_{m}(1)]\slash \sqrt{\;\overline{w}}.$$

We now introduce several models of random graphs specific to our needs. Ultimately, we wish to model our graph $\Upsilon$ by a suitable random graph model $\Dpp.$ However, the random graph $\Dpp$ is somewhat complicated, and so, we build up to understanding the eigenvalues of this graph by first passing through some simpler models of random graphs. In particular, $\Dpp$ is formed from the union of $\Gpp$, $\Bpp$ and $\Lpp$; we study each of these graphs in turn.

In what follows, $\pi$ will be a vector in $(0,\infty)^{\vert V\vert}$, where $V$ will be the vertex set of the graph under consideration. Informally, $\pi_{i}$ corresponds to the `weight' of the $ith$ vertex of the graph. For notational ease we will write $\pi(v)$ to be the entry of $\pi$ corresponding to the vertex $v$. Furthermore, to be technically accurate, the below constructions are actually constructions of sequences of graphs. We therefore require a choice of $\pi_{m}\in\mathbb{Z}^{m}$ for each $m$: we will ignore the subscripts $m$ for notational ease. For the remainder of this text, we \emph{always} assume that $p<1\slash 2$.
\begin{definition}[{\bf The graph }$\boldsymbol{\Gpp}$]
Let $m\geq 1$, $\pi_{min}$ be a constant independent of $m$, and let $\pi\in[\pi_{min},\infty)^{m}$. The random graph $\Gpp$ has vertex set $V=\{v_{1},\hdots,v_{m}\}$. The map $*:V\rightarrow V$ a fixed-point free involution specified in advance. Let $v_{i}^{-1}:=*(v_{i})$. Each edge $(v_{i},v_{j})$ is added with probability $(\pi (v_{i})\pi(v_{j}^{-1})+\pi(v_{i}^{-1})\pi(v_{j}))p$.
\end{definition}

\begin{definition}[{\bf The graph} $\boldsymbol{\Bpp}$]
Let $m_{1},m_{2}\geq 1$, $\pi_{min}>0$ be a constant independent of $m_{1}$ and $m_{2}$, and $\pi\in[\pi_{min},\infty)^{m_{1}+m_{2}}$. The random graph $\Bpp$ is the bipartite graph with vertex partition $V_{1}\sqcup V_{2},$ where $V_{1}=\{u_{1},\hdots,u_{m_{1}}\}$ and $ V_{2}=\{v_{1},\hdots ,v_{m_{2}}\}$. Each edge $(u_{i},v_{j})$ is added with probability $\pi (u_{i})\pi(v_{j})p$.
\end{definition}
\begin{definition}[{\bf The graph} $\boldsymbol{\Lpp}$]
Let $m_{1},m_{2}\geq 1$, $\pi_{min}>0$ be a constant independent of $m_{1}$ and $m_{2}$, and let $\pi\in[\pi_{min},\infty)^{m_{1}+m_{2}}$. The random graph $\Lpp$ is defined as follows. It has vertex set $V=V_{1}\sqcup V_{2}\sqcup V_{3}$, where  $V_{1}=\{u_{1},\hdots,u_{m_{1}}\}$, $V_{2}=\{v_{1},\hdots,v_{m_{2}}\},$ and $V_{3}=\{v_{1}^{-1},\hdots,v_{m_{2}}^{-1}\}.$ The map $*:V_{1}\rightarrow V_{1}$ is a fixed-point free involution specified in advance. Let $u_{i}^{-1}:=*(u_{i})$. Each edge $(u_{i},v_{j})$ is added with probability $\pi(u_{i}^{-1})\pi(v_{j})p$. Each edge $(u_{i},v_{j}^{-1})$ is added with probability $\pi(u_{i})\pi(v_{j})p$.
\end{definition}
\begin{definition}[{\bf The graph} $\boldsymbol{\Dpp}$]
Let $m_{1},m_{2}\geq 1$, $\pi_{min}>0$ be a constant independent of $m_{1}$ and $m_{2}$, and let $\pi\in[\pi_{min},\infty)^{m_{1}+m_{2}}$. The random graph $\Dpp$ is the random graph defined as follows. It has vertex set $V=V_{1}\sqcup V_{2}\sqcup V_{3}$, where  $V_{1}=\{u_{1},\hdots,u_{m_{1}}\}$, $V_{2}=\{v_{1},\hdots,v_{m_{2}}\},$ and $V_{3}=\{v_{1}^{-1},\hdots,v_{m_{2}}^{-1}\}.$ The map $*:V_{1}\rightarrow V_{1}$ is a fixed-point free involution specified in advance. Let $u_{i}^{-1}:=*(u_{i})$.
\begin{enumerate}[label={$\arabic* )$}]
\item Each edge $(u_{i},u_{j})$ is added with probability $$(\pi (u_{i})\pi(u_{j}^{-1})+\pi(u_{i}^{-1})\pi(u_{j}))p.$$
\item Each edge $(u_{i},v_{j})$ is added with probability $\pi(u_{i}^{-1})\pi(v_{j})p$.
\item Each edge $(u_{i},v_{j}^{-1})$ is added with probability $\pi(u_{i})\pi(v_{j})p$.
\item Finally, each edge $(v_{i},v_{j}^{-1})$ is added with probability $\pi(v_{i})\pi(v_{j})p$.
\end{enumerate}
\end{definition}

\begin{definition}
For all of the above graphs, we define $$\pi_{max}=\max\{\pi(v)\;:\;v\in V\}.$$
\end{definition}

\subsection{Switching between the adjacency matrix and expected adjacency matrix}
Define the matrix $$\MGpp = \mathbb{E}(D(\Sigma))^{-1}\mathbb{E}(A(\Sigma))$$ for $\Sigma\sim\Gpp.$ We define similarly the matrices $M(\Gamma(m,\underline{w}))$,\\ $\MBpp,\;\MLpp,$ and $\MDpp.$

It will be important for us to switch between the random matrices $D^{-1}A$ associated to a random graph and the fixed matrix $\mathbb{E}(D^{-1})\mathbb{E}(A).$  In order to do this, we will use the matrix Bernstein inequality. 
\begin{theorem}\label{thm: Matrix bernstein}(\cite[Theorem 1.4]{tropp2012user}, c.f.  \cite{oliveira2009concentration})
Let $X_{1},\hdots ,X_{M}$ be a set of independent random symmetric $m\times m$ matrices. Suppose for $i=1,\hdots M$: $ \vert\vert X_{i}-\mathbb{E}(X_{i})\vert\vert_{2}\leq L$.
Let $X=\sum_{i}X_{i}$, and $$\sigma(X)=\left\vert \left\vert\sum_{i}\mathbb{E}\left[(X_{i}-\mathbb{E}(X_{i}))(X_{i}-\mathbb{E}(X_{i}))^{T}\right]\right\vert\right\vert.$$
Then for all $t>0$:
$$\mathbb{P}(\vert\vert X-\mathbb{E}(X)\vert\vert_{2}\geq t)\leq 2m\exp\left\{\dfrac{-t^{2}\slash 2}{\sigma(X)+Lt\slash 3}\right\}.$$
\end{theorem}

 We will also need to use the Chernoff bounds: for $X\sim Bin(m,p)$ and $\sigma\in [0,1],$ $$\mathbb{P}(\vert X - mp\vert \geq \delta mp)\leq 2\exp(-mp \delta^{2}\slash 3).$$ 
This allows us to prove the following.
\begin{lemma}\phantom{x}
\begin{enumerate}[label=$\roman*)$]
    \item Let $\Sigma\sim\Gpp$ with $mp=\Omega_{m}(\pi_{max}\log m)$ or $\Sigma\sim\Gw$ with $w_{min}p=\Omega_{m} (w_{max}\log m)$. Then
$$\vert\vert D(\Sigma)^{-1}A(\Sigma)-\mathbb{E}(D(\Sigma)^{-1})\mathbb{E}(A(\Sigma))\vert \vert_{2}=o_{m}(1)$$
with probability tending to $1$ as $m$ tends to infinity.

\item Let $\Sigma\sim \Bpp$ with $\min\{m_{1},m_{2}\}p=\Omega_{m_{1}}(\pi_{max}\log (m_{1}+m_{2}))$. Then
$$\vert\vert D(\Sigma)^{-1}A(\Sigma)-\mathbb{E}(D(\Sigma)^{-1})\mathbb{E}(A(\Sigma))\vert \vert_{2}=o_{m_{1}}(1)$$
with probability tending to $1$ as $m_{1}$ tends to infinity.
\end{enumerate}
\end{lemma}
\begin{proof}
Consider statement $i)$. Let $\Sigma\sim\Gpp$. For $1\leq i<j\leq m$ let $E^{i,j}$ be the matrix with all zero entries except for entry $1$ in position $(i,j)$. Let $X_{i,j}$ be the random symmetric matrix $X_{i,j}=\xi_{i,j}(E^{i,j}+E^{j,i})$, where $\xi_{i,j}\sim Bernoulli(p_{i,j})$ for $p_{i,j}:=(\pi(u_{i})\pi(u_{j}^{-1})+\pi(u_{i}^{-1})\pi(u_{j}))p.$
For each $i,j;$ $$\vert\vert X_{i,j}-\mathbb{E}(X_{i,j})\vert\vert_{2}\leq \max \left\{\vert\vert X_{i,j}-\mathbb{E}(X_{i,j})\vert \vert_{1},\vert\vert X_{i,j}-\mathbb{E}(X_{i,j})\vert \vert_{\infty}\right\}\leq1.$$
Furthermore, for each $i,j$; $$ \mathbb{E}[(X_{i,j}-\mathbb{E}(X_{i,j}))(X_{i,j}-\mathbb{E}(X_{i,j}))^{T}]= p_{i,j}(1-p_{i,j})(E^{i,j}+E^{j,i}).$$

Let $X=\sum_{i,j}X_{i,j}$, so that $X=A(\Sigma)$.
Since $p_{i,j}\leq \pi_{max}2p$, we have that 
\begin{align*}
\sigma(X)&= \left\vert\left\vert\sum_{i,j}p_{i,j}(1-p_{i,j})(E^{i,j}+E^{j,i}) \right\vert \right\vert_{2}\\
&\leq \max \left\{\left\vert\left\vert\sum_{i,j}p_{i,j}(1-p_{i,j})(E^{i,j}+E^{j,i}) \right\vert \right\vert_{1},\left\vert\left\vert\sum_{i,j}p_{i,j}(1-p_{i,j})(E^{i,j}+E^{j,i}) \right\vert \right\vert_{\infty}\right\}\\&\leq 2\pi_{max}mp.
\end{align*}
Therefore, by Theorem \ref{thm: Matrix bernstein} (the matrix Bernstein inequality), letting $\omega$ be any function such that $
\omega=\Omega_{m}(1)$ and $mp=\Omega_{m}(\pi_{max}\omega \log m) :$
\begin{align*}
    \mathbb{P}\left(\vert \vert X-\mathbb{E}(X)\vert\vert_{2}\geq \sqrt{\pi_{max}mp\omega \log m}\right)&\leq \exp\left\{\dfrac{-\pi_{max}mp\omega \log m}{2mp\pi_{max}+\sqrt{mp\omega \log m}\slash 3}\right\}\\
    &\leq m^{-\omega\slash 100}\\
    &=o_{m}(1).
\end{align*}
Next, $\mathbb{E}(D)_{min}\geq \pi_{min}^{2}mp$. Note that $X=A(\Sigma),$ and so:
\begin{align*}
  &  \mathbb{P}\left(\vert \vert \mathbb{E}(D)_{min}^{-1}[A(\Sigma)-\mathbb{E}(A(\Sigma))]\vert\vert_{2} \geq \sqrt{\pi_{max}mp\omega \log m}\slash \pi_{min}mp\right)\\ &\leq  \mathbb{P}\left(\vert \vert A(\Sigma)-\mathbb{E}(A(\Sigma)\vert\vert_{2}\geq \sqrt{\pi_{max}mp\omega \log m}\right)\\&=o_{m}(1).
\end{align*}
Finally, since $mp=\Omega_{m}(\pi_{max}\log m)$, it is follows by a standard application of the Chernoff bounds (see e.g. \cite[Theorem 3.4]{frieze_karonski}) that with probability tending to $1$,
$$\vert \vert I-D(\Sigma)\mathbb{E}(D(\Sigma))^{-1}\vert\vert_{2}\leq o_{m}(1),$$
i.e. the actual degree of a vertex is close to its expected degree. Hence

\begin{align*}
    \vert \vert D^{-1}A(\Sigma)-\mathbb{E}(D)^{-1}A(\Sigma)\vert \vert_{2}&\leq \vert \vert I-\mathbb{E}(D)^{-1}D\vert \vert_{2}\vert \vert D^{-1}A(\Sigma)\vert \vert_{2}\\
    &\leq o_{m}(1)
\end{align*}
with probability $1-o_{m}(1)$, and the result follows. The remaining cases follow similarly.
\end{proof}
As $A(\Lpp)$ and $A(\Dpp)$ are formed by summing matrices of the above form, this immediately implies the following.

\begin{lemma}\label{lem: difference between adjacency matrix and expected adjacency matrix}
Let $\Sigma\sim \Lpp$ or $\Sigma\sim \Dpp$ where $(m_{1}+2m_{2})p=\Omega_{m}(\pi_{max}\log (m_{1}+2m_{2}))$. Then
$$\vert\vert D(\Sigma)^{-1}A(\Sigma)-\mathbb{E}(D(\Sigma)^{-1})\mathbb{E}(A(\Sigma))\vert \vert_{2}=o(1)$$
with probability tending to $1$ as $m_{1}+m_{2}$ tends to infinity.
\end{lemma}

\subsection{The first eigenvalue of \texorpdfstring{$\boldsymbol{\Gpp}$}{Gamma(m,pi,p)}}
We now turn to analysing the first eigenvalue of the graph $\Gpp$.
\begin{lemma}\label{lem: eigenvalue of MGpp}
Suppose that $m\geq 1$, and $p$ is such that that $mp=\Omega_{m}\left(\pi_{max}\log^{6}m\right).$ Then $ \mu_{2}(\MGpp)=o_{m}(1).$
\end{lemma}
\begin{proof}
Let $M=\MGpp$, and let $W$ be the expected degree matrix of $\Gpp$. Define \begin{itemize}
    \item  $N_:=\sum_{v\in V}\pi(v),$
    \item  $\underline{v}(x_{i})=\pi(x_{i})Np$, 
    \item $\underline{v}'(x_{i})=\pi(x_{i}^{-1})Np,$ and 
    \item $\underline{w}(x_{i})=2(\pi(x_{i})+\pi(x_{i}^{-1}))Np$.
\end{itemize}
Let $W_{1}=\mathbb{E}(D(\Gamma(m,\underline{v})))$ and $W_{2}=\mathbb{E}(D(\Gamma(m,\underline{v}'))$. Define similarly the matrices $A_{1}$ and $A_{2}$ as the expected adjacency matrices of the graphs $\Gamma(m,\underline{v})$ and $\Gamma (m,\underline{v'})$. Let $W'=\mathbb{E}(D(\Gamma(m,\underline{w}))$. Note that $(WW_{1}^{-1})_{max}\leq 1\slash \pi_{min}.$

Let $M'=M+W^{-1}A_{1}+W^{-1}A_{2}=2M(\Gamma(m,\underline{w})).$ By our assumptions, we have that $w_{min}=\Omega (w_{max}\sqrt{\overline{w}}\log^{3}(m))$, and similarly for $\underline{v}$ and $\underline{v}'$. Therefore $\mu_{2}(M')=o_{m}(1)$ by \cite[Theorem 5]{chungrandomgraph} and Lemma \ref{lem: difference between adjacency matrix and expected adjacency matrix}. Similarly $$\mu_{1}(-W_{1}^{-1}A_{1})=-\mu_{m}(W_{1}^{-1}A_{1})=o_{m}(1),$$ and $\mu_{2}(-W_{2}^{-1}A_{2})=o_{m}(1)$.

By Weyl's inequality and the Courant-Fischer theorem, we have that
\begin{align*}
    \mu_{2}(M)&=\mu_{2}\bigg(M'-W^{-1}W_{1}(W_{1}^{-1}A_{1})-W^{-1}W_{2}(W_{2}^{-1}A_{2})\bigg)\\
    &\leq \mu_{2}(M')+[W^{-1}W_{1}]_{max}\mu_{1}(-W_{1}^{-1}A_{1})+[W^{-1}W_{2}]_{max}\mu_{1}(-W_{2}^{-1}A_{2})\\
    &\leq \mu_{2}(M')+\mu_{1}(-W_{1}^{-1}A_{1})+\mu_{1}(-W_{2}^{-1}A_{2})\\
    & \leq o_{m}(1).
\end{align*}
\end{proof}

\subsection{The first eigenvalue of \texorpdfstring{$\boldsymbol{\Bpp}$}{B(m,pi,p)}}

We now aim to prove a similar theorem for $\Bpp$. In particular, we prove the following.

\begin{lemma}\label{lem: eigenvalue of MBpp}
Suppose that $m_{1},m_{2}\geq 1$, and $p$ is such that that $$\min\{m_{1},m_{2}\}p=\Omega_{m_{1}}\left(\pi_{max}\log^{6}(m_{1}+m_{2})\right).$$ Then $ \mu_{2}(\MBpp )\leq o_{m_{1}}(1).$
\end{lemma}

\begin{proof}
The proof proceeds as in the proof of Theorem \ref{lem: eigenvalue of MGpp}. Let $A_{1}$ be the matrix with $(A_{1})_{i,j}=\pi(u_{i})\pi (u_{j})p$ for $i,j\leq m$ and $0$ otherwise. Let $A_{2}$ be the matrix with $(A_{2})_{i,j}=\pi(v_{i})\pi(v_{j})p$ for $i,j> m$ and $0$ otherwise.

Then $M'=M+W^{-1}A_{1}+W^{-2}A_{2}=2M(\Gamma(m_{1}+m_{2},\underline{w})),$ where $\underline{w}$ is defined similarly to the previous lemma. We observe $$\mu_{2}(M'),\mu_{1}(-W^{-1}A_{1}),\mu_{1}(-W^{-1}A_{2})=o_{m_{1}+m_{2}}(1).$$ The proof now follows similarly to Lemma \ref{lem: eigenvalue of MGpp}.
\end{proof}

\subsection{The first eigenvalue of \texorpdfstring{$\boldsymbol{\Lpp}$}{Lambda(m1,m2,pi,p)}}

We now turn to analysing the eigenvalues of $\Lpp$. 
\begin{lemma}\label{lem: eigenvalue of MLpp}
Suppose that $m_{1},m_{2}\geq 1$, and $p$ is such that that $$\min\{m_{1},m_{2}\}p=\Omega_{m_{1}}\left(\pi_{max}\log^{6}(m_{1}+m_{2})\right).$$  Then $$ \mu_{2}(\MLpp )\leq o_{m_{1}}(1).$$
\end{lemma}
\begin{proof}
We will use the Rayleigh quotient theorem. We note that
\begin{align*}
   M:= \MLpp=W_{1}M_{1}+W_{2}M_{2},
\end{align*}
where $M_{1}$ corresponds to the edges in $V_{1}\times V_{2}$ and $M_{2}$ corresponds to the edges in $V_{1}\times V_{3}$ (normalised by the expected degrees). In particular $M_{1}=M(\mathcal{B}(m_{1},m_{2},\pi,p))$, and $M_{2}=M(\mathcal{B}(m_{1},m_{2},\pi\circ *,p))$. Each $M_{i}$ has largest eigenvalue equal to $1$, with corresponding eigenvector $\vect{1}_{V_{1}\sqcup V_{2}}$ and $\vect{1}_{V_{1}\sqcup V_{3}}$ respectively. Similarly, the first eigenvalue of $M$ is $1$, with corresponding eigenvector $\vect{1}_{V_{1}\sqcup V_{2}\sqcup V_{3}}$.

Therefore, to estimate $\mu_{2}(M)$, we analyse
$$\mu_{2}:=\max\limits_{\substack{\vect{x}\perp \vect{1}_{V}\\ \vert \vert \vect{x}\vert \vert =1}}\langle M\vect{x},\vect{x}\rangle.$$ Let $\vect{\phi}$ be a unit vector with $\vect{\phi}\cdot \vect{1}_{V}=0$ achieving the above maximum, whose existence is guaranteed by the Courant-Fischer theorem. We may write 
\begin{align*}
    \vect{\phi}&=\vect{\xi}+\vect{\psi}\\
    &=\begin{pmatrix}
    \alpha \vect{1}_{V_1}\\
    \beta \vect{1}_{V_{2}}\\
    \gamma \vect{1}_{V_{3}}
    \end{pmatrix}+\begin{pmatrix}
 \vect{\psi}_{1}\\
 \vect{\psi}_{{2}}\\
\vect{\psi}_{{3}}
    \end{pmatrix},
\end{align*}
where $\vect{\psi}_{i}\cdot \vect{1}_{V_{i}}=0$ for $i=1,2,3.$ Since $\vect{\phi}\cdot\vect{1}_{V}=0,$ we have that $$\alpha \vert V_{1}\vert+\beta\vert V_{2}\vert+\gamma \vert V_{3}\vert=\alpha m_{1}+(\beta +\gamma)m_{2}=0,$$ and so $\alpha(\beta+\gamma)\leq 0.$ 

Define $$N_{1}=\sum_{u\in V_{1}}\pi (u),\mbox{ and }N_{2}=\sum_{v\in V_{2}}\pi (v).$$ Furthermore, define 
\begin{align*}
D_{1}&=diag\bigg(\pi(u_{1}^{-1})N_{2},\hdots,\pi(u_{m_{1}}^{-1})N_{2},\underbrace{0,\hdots ,0}_{m_{2}},\pi(v_{1})N_{1},\hdots, \pi(v_{m_{2}})N_{1}\bigg),
\\D_{2}&=diag\bigg(\pi(u_{1})N_{2},\hdots,\pi(u_{m_{1}})N_{2},\pi(v_{1})N_{1},\hdots,\pi(v_{m_{2}})N_{1},\underbrace{0,\hdots ,0}_{m_{2}}\bigg).\end{align*}
Then $$M_{1}=D_{1}^{-1}\begin{pmatrix}
0&A_{1}&0\\
A_{1}^{T}&0&0\\
0&0&0,
\end{pmatrix}$$
where $(A_{1})_{i,j}=\pi(u_{i}^{-1})\pi(v_{j})$, and similarly $$M_{2}=D_{2}^{-1}\begin{pmatrix}
0&0&A_{2}\\
0&0&0\\
A_{2}^{T}&0&0,
\end{pmatrix},$$
where $(A_{2})_{i,j}=\pi(u_{i})\pi(v_{j})$.
We can see that 
\begin{align*}
     W_{1}&=diag\bigg(\dfrac{\pi(u_{1}^{-1})}{\pi(u_{1})+\pi(u_{1}^{-1})},\hdots,\dfrac{\pi(u_{m_{1}}^{-1})}{\pi(u_{m_1})+\pi(u_{m_1}^{-1})},\underbrace{0,\hdots ,0}_{m_{2}},\underbrace{1,\hdots, 1}_{m_{2}}\bigg),
\end{align*}
 so that $\vert \vert W_{1}\vert\vert_{2}\leq 1$. Similarly, \begin{align*}
     W_{2}&=diag\bigg(\dfrac{\pi(u_{1})}{\pi(u_{1})+\pi(u_{1}^{-1})},\hdots,\dfrac{\pi(u_{m_{1}})}{\pi(u_{m_1})+\pi(u_{m_1}^{-1})},\underbrace{1,\hdots, 1}_{m_{2}},\underbrace{0,\hdots ,0}_{m_{2}}\bigg),,
\end{align*} so that $\vert\vert W_{2}\vert\vert_{2}\leq 1$. Now, $\vert \vert \vect{\psi}\vert\vert \leq 1$ and $\vect{\psi}\cdot \vect{1}_{V_{1}\sqcup V_{2}}=0$, and so by the Courant-Fischer theorem:
$$\langle M_{1}\vect{\psi},\vect{\psi}\rangle\leq \mu_{2}(M_{1})\vert\vert \vect{\psi}\vert\vert^{2}\leq \mu_{2}(M_{1}).$$
Therefore:
\begin{align*}
    \langle W_{1}M_{1}\vect{\psi},\vect{\psi}\rangle &=\vect{\psi}^{T} W_{1}M_{1}\vect{\psi}\\
    &=\vect{\psi}^{T} W_{1}\vect{\psi}\vect{\psi}^{T}M_{1}\vect{\psi}\slash \vert\vert\vect{\psi}\vert\vert^{2}\\ &= \vect{\psi}^{T}M_{1}\vect{\psi}\left(\vect{\psi}^{T}W_{1}\vect{\psi}^{T}\slash \vect{\psi}^{T}\vect{\psi}\right)\\
    &\leq \vect{\psi}^{T}M_{1}\vect{\psi}\max_{\vect{x}\neq \vect{0}}\left(\vect{x}^{T}W_{1}\vect{x}^{T}\slash \vect{x}^{T}\vect{x}\right)\\
    &\leq\langle M_{1}\vect{\psi},\vect{\psi}\rangle \left(\vert \vert W_{1}\vert \vert_{2}\right)^{2}\\
    &\leq  \mu_{2}(M_{1}) \vert\vert \vect{\psi}\vert\vert^{2}\left(\vert \vert W_{1}\vert \vert_{2}\right)^{2}\\
    &\leq  \mu_{2}(M_{1}) \left(\vert \vert W_{1}\vert \vert_{2}\right)^{2}\\
    &\leq \mu_{2}(M_{1})=o_{m_{1}}(1).
\end{align*}

We can also see that $\langle W_{2}M_{2}\vect{\psi},\vect{\psi}\rangle\leq o_{m_{1}}(1).$ It therefore remains to analyse $\langle M\vect{\xi},\vect{\xi}\rangle.$ We have that 
\begin{align*}
M\vect{\xi}&=\begin{pmatrix}
\begin{pmatrix}
\dfrac{\gamma \pi (u_{1})+\beta \pi(u_{1}^{-1})}{\pi(u_{1})+\pi(u_{1}^{-1})}\\
\vdots\\
\dfrac{\gamma \pi (u_{m_1})+\beta \pi (u_{m_1}^{-1})}{\pi(u_{m_1})+\pi(u_{m_1}^{-1})}\\
\end{pmatrix}\\
\alpha\vect{1}_{V_2}\\
\alpha \vect{1}_{V_{3}}
\end{pmatrix}.
\end{align*}
Therefore, as $\vert V_{1}\vert ,\vert V_2\vert =\vert V_3\vert\geq 0$ and $*$ is fixed-point free:
\begin{align*}
    \langle M\vect{\xi},\vect{\xi}\rangle &=\sum_{u\in V_{1}}\dfrac{\alpha\gamma\pi(u)+\alpha\beta\pi(u^{-1})}{\pi(u)+\pi(u^{-1})} +\alpha\beta \vert V_2\vert +\alpha \gamma \vert V_3\vert\\
&=\alpha \gamma\vert V_{1}\vert+\alpha \beta \vert V_{1}\vert +\alpha\beta \vert V_2\vert +\alpha \gamma \vert V_3\vert\\ 
    &=\alpha(\beta +\gamma)(\vert V_{1}\vert +\vert V_2\vert)\leq 0.
\end{align*}
 Hence
\begin{align*}
    \mu_{2}(M)=\langle M\vect{\phi},\vect{\phi}\rangle
    &=\langle M\vect{\xi},\vect{\xi}\rangle+\langle W_{1}M_{1}\vect{\psi},\vect{\psi}\rangle+\langle W_{2}M_{2}\vect{\psi},\vect{\psi}\rangle\\
    &\leq \langle W_{1}M_{1}\vect{\psi},\vect{\psi}\rangle+\langle W_{2}M_{2}\vect{\psi},\vect{\psi}\rangle\\
    &\leq o_{m_{1}}(1).
\end{align*}
\end{proof}

\subsection{The first eigenvalue of \texorpdfstring{$\boldsymbol{\Dpp}$}{Delta(m1,m2,pi,p)}}

We finally have all the pieces in place to analyse the first eigenvalue of the random graph $\Dpp$. We will approach this in a similar manner to the previous subsection. However, we need three different results: we split the subsequent analysis into three smaller sections.

\subsubsection{The case \texorpdfstring{$\boldsymbol{m_{1}=\Omega_{m_{1}}(m_{2})}$}{m1=Omega(m2)}}
\begin{lemma}\label{lem: eigenvalue of Dpp m1=Omega(m2)}
Let $m_{1},m_{2}\geq 1$ with $m_{1}=\Omega_{m_{1}}(\pi_{max}m_{2})$. Suppose that $p$ satisfies $(m_{1}+2m_{2})p=\Omega_{m_{1}}(\pi_{max}\log^{6}(m_{1}+2m_{2}))$. Let $\Sigma\sim\Dpp$. Then with probability tending to $1$ as $m_{1}$ tends to infinity, 
$$\lambda_{1}(\Sigma)\geq 1-o_{m_{1}}(1).$$
\end{lemma}

\begin{proof}
Since $m_{1}=\Omega_{m_{1}}(\pi_{max}m_{2})$, we have that $m_{1}p=\Omega_{m_{1}}(\pi_{max}\log^{6}(m_{1})).$ Let $\Sigma_{1}$ be the subgraph of $\Sigma$ corresponding to the edges in $V_{1}^{2}$, so that $\Sigma_{1}\sim \Gpp$. Similarly let  $\Sigma_{2}$ be the subgraph of $\Sigma$ corresponding to the edges in $V_{1}\times (V_{2}\sqcup V_{3})$, so that $\Sigma_{2}\sim \Lpp$. Finally, let $\Sigma_{3}$ be the subgraph of $\Sigma$ corresponding to the edges in $V_{2}\times V_{3}$, so that $\Sigma_{3}\sim \Bpp$.

We note that $D(\Sigma)_{min}\geq \pi_{min}^{2}m_{1}p\slash 2$ almost surely. Furthermore, with probability tending to $1$, 

$$\max_{v\in V_{1}(\Sigma_{2})}deg_{\Sigma_{2}}(v) =o_{m_{1}}(m_{1}p\slash \pi_{max}),$$
since $m_{1}=\Omega_{m_{1}}(m_{2}\pi_{max}).$ Similarly, for some fixed $C>0$, with probability tending to $1$: $$\max_{v\in V_{2}(\Sigma_{2})}deg_{\Sigma_{2}}(v) \leq C\pi_{max}m_{1}p.$$
By Lemma \ref{lem: max eigenvalue of adjacency matrix}, almost surely $$\max_{v\in V(\Sigma_{3})}deg_{\Sigma_{3}}(v)=o_{m_{1}}(m_{1}p\slash \pi_{max}).$$
We have that $A(\Sigma)=A(\Sigma_{1})+A(\Sigma_{2})+A(\Sigma_{3}),$ and by Lemma \ref{lem: max eigenvalue of adjacency matrix} we see that: $$\vert \vert A(\Sigma_{2})\vert\vert_{2},\vert \vert A(\Sigma_{3})\vert\vert_{2}=o_{m_{1}}(m_{1}p)$$ with probability tending to $1$. Therefore, $$\vert \vert D(\Sigma)^{-1}A(\Sigma_{2})\vert\vert_{2},\vert \vert D(\Sigma)^{-1}A(\Sigma_{3})\vert\vert_{2}=o_{m_{1}}(1)$$ almost surely. Hence $$\vert \vert D(\Sigma)^{-1}A(\Sigma)-D(\Sigma)^{-1}A(\Sigma_{1})\vert\vert_{2}=o_{m_{1}},$$ so that $\mu_{i}(D(\Sigma)^{-1}A(\Sigma))=\mu_{i}(D(\Sigma)^{-1}A(\Sigma_{1}))+o_{m_{1}}(1)$ for all $i$. By applying the Chernoff bounds, we see that almost surely 
$$\vert \vert D(\Sigma)\vert_{\Sigma_{1}}-D(\Sigma_{1})\vert \vert =o_{m_{1}}\min\left\{\vert \vert D(\Sigma)\vert_{\Sigma_{1}}\vert \vert ,\vert \vert D(\Sigma_{1})\vert \vert \right\},$$ so that
$$\vert \vert D(\Sigma)^{-1}A(\Sigma_{1})-D(\Sigma_{1})^{-1}A(\Sigma_{1})\vert\vert_{2}=o_{m_{1}}(1)$$
with probability tending to $1$. By assumption, $\Sigma_{1}\sim\Gpp$ with $m_{1}p=\Omega(\pi_{max}\log^{6}m_{1}),$ and hence $\mu_{2}(D(\Sigma_{1})^{-1}A(\Sigma_{1}))=o_{m_{1}}(1)$ by Lemmas \ref{lem: difference between adjacency matrix and expected adjacency matrix} and \ref{lem: eigenvalue of MGpp}. Hence $$\mu_{2}(D(\Sigma)^{-1}A(\Sigma))=o_{m_{1}}(1),$$ almost surely and the result follows.

\end{proof}
\subsubsection{The case o\texorpdfstring{$\boldsymbol{m_{2}=\Omega_{m_{2}}(m_{1})}$}{m1=Omega(m2)}}
\begin{lemma}\label{lem: eigenvalue of Dpp m2=Omega(m1)}
Let $m_{1},m_{2}\geq 1$ with $m_{2}=\Omega(\pi_{max}m_{1})$. Suppose that $p$ satisfies $(m_{1}+2m_{2})p=\Omega_{m_{2}}(\pi_{max}\log^{6}(m_{1}+2m_{2}))$. Let $\Sigma\sim\Dpp$. Then with probability tending to $1$ as $m_{2}$ tends to infinity, 
$$\lambda_{1}(\Sigma)\geq 1-o_{m_{2}}(1).$$
\end{lemma}

\begin{proof}
Since $m_{2}=\Omega_{m_{2}}(\pi_{max}m_{1})$, we have that $m_{2}p=\Omega_{m_{2}}(\pi_{max}\log^{6}(m_{2})).$ Let $\Sigma_{1}$ be the subgraph of $\Sigma$ corresponding to the edges in $V_{1}^{2}$, so that $\Sigma_{1}\sim \Gpp$. Similarly let  $\Sigma_{2}$ be the subgraph of $\Sigma$ corresponding to the edges in $V_{1}\times (V_{2}\sqcup V_{3})$, so that $\Sigma_{2}\sim\Lpp$. Finally, let $\Sigma_{3}$ be the subgraph of $\Sigma$ corresponding to the edges in $V_{2}\times V_{3}$, so that $\Sigma_{3}\sim \Bpp$.

We note that $D(\Sigma)_{min}\geq \pi_{min}^{2}m_{2}\slash 2$ almost surely. Furthermore, we may deduce similarly to the previous lemma that, with probability tending to $1$: $$\vert \vert A(\Sigma_{1})\vert\vert_{2},\vert \vert A(\Sigma_{2})\vert\vert_{2}=o_{m_{2}}(m_{2}p\slash \pi_{max}).$$
We have that $A(\Sigma)=A(\Sigma_{1})+A(\Sigma_{2})+A(\Sigma_{3}).$ Therefore $$\vert \vert D(\Sigma)^{-1}A(\Sigma_{1})\vert\vert_{2},\vert \vert D(\Sigma)^{-1}A(\Sigma_{2})\vert\vert_{2}=o_{m_{2}}(1)$$ almost surely. Hence $$\vert \vert D(\Sigma)^{-1}A(\Sigma)\vert_{\Sigma_{3}}-D(\Sigma)^{-1}A(\Sigma_{3})\vert\vert_{2}=o_{m_{2}}(1)$$ almost surely, so that $\mu_{i}(D(\Sigma)^{-1}A(\Sigma))=\mu_{i}(D(\Sigma)^{-1}A(\Sigma_{3}))+o_{m_{2}}(1)$ for all $i$. Finally, 
$$\vert \vert D(\Sigma)^{-1}A(\Sigma_{3})-D(\Sigma_{3})^{-1}A(\Sigma_{3})\vert\vert_{2}\leq o_{m_{2}}(1)$$
with probability tending to $1$. By assumption, $\Sigma_{3}\sim\Bpp$ with $m_{2}p=\Omega_{m_{2}}(\pi_{max}\log^{6}m_{2}),$ so that $\mu_{2}(D(\Sigma_{3})^{-1}A(\Sigma_{3}))\leq o_{m_{2}}(1)$ by Lemmas \ref{lem: difference between adjacency matrix and expected adjacency matrix} and \ref{lem: eigenvalue of MBpp}. Therefore $$\mu_{2}(D(\Sigma)^{-1}A(\Sigma))\leq o_{m_{2}}(1),$$
as required.

\end{proof}
\subsubsection{The case \texorpdfstring{$\boldsymbol{m_{1}=\Theta ( m_{2})}$}{m1=Theta(m2)}}
We now turn to the final case.

\begin{lemma}\label{lem: eigenvalue of MDpp m1=theta(m2)}
Let $m_{1},m_{2}\geq 1$ with $$O_{m_{1}+m_{2}}(m_{2}\slash \pi_{max})\leq m_{1}\leq O_{m_{1}+m_{2}}(m_{2}\pi_{max}).$$ Suppose that $p$ satisfies $(m_{1}+2m_{2})p=\Omega_{m_{1}}(\pi_{max}^{2}\log^{6}(m_{1}+2m_{2}))$. Then 
$$\mu_{2}\left(\MDpp\right)\leq o_{m_{1}}(1).$$
\end{lemma}
\begin{proof}
We may write
\begin{align*}
    M:&=\MDpp=W_{1}M_{1}+W_{2}M_{2}+W_{3}M_{3},
\end{align*}
where $M_{1}=M(\Gpp)$, $M_{2}=M(\Lpp)$, and\\ $M_{3}=M(\Bpp)$.

Now, as previously, let $\vect{\phi}$ be a vector with $\norm{\phi}=1,$ $\vect{\phi}\cdot \vect{1}_{V}=0,$ and $$\mu_{2}(M)=\langle M\vect{\phi},\vect{\phi}\rangle.$$ This vector is guaranteed to exist by the Courant-Fischer theorem. We again write \begin{align*}
    \vect{\phi}&=\begin{pmatrix}
    \alpha\vect{1}_{V_{1}}\\
    \beta\vect{1}_{V_{2}\sqcup V_{3}}
    \end{pmatrix}+\begin{pmatrix}
   \vect{\psi}_{1}\\
   \vect{\psi}_{2}
    \end{pmatrix}\\
    &=\vect{\xi}+\vect{\psi},
\end{align*}
where $\vect{\psi}_{1}\cdot \vect{1}_{V_{1}}=\vect{\psi}_{2}\cdot \vect{1}_{V_{2}\sqcup V_{3}}=0.$
As we previously argued, by our assumptions on $p$ we have that:
$$\langle W_{i}M_{i}\vect{\psi},\vect{\psi}\rangle=o_{m_{1}}(1),$$
for $i=1,2,3.$

It therefore remains to analyse $\langle M\vect{\xi},\vect{\xi}\rangle.$ Recall that we defined $$N_{1}=\sum_{u\in V_{1}}\pi (u)\;N_{2}=\sum_{v\in V_{2}}\pi (v).$$ 
For $i\leq m_{1}$, we calculate $$M_{i,j}:=
\begin{cases}
\dfrac{\pi (u_{i})\pi(u_{j}^{-1})+\pi (u_{i}^{-1})\pi(u_{j})}{{(\pi(u_{i})+\pi(u_{i}^{-1}))(N_{1}+N_{2})}}: \;j\leq m_{1},\\
\dfrac{ \pi(u_{i}^{-1})\pi(v_{j})}{(\pi(u_{i})+\pi(u_{i}^{-1}))(N_{1}+N_{2})}: \;m_{1}+1\leq j\leq m_{1}+m_{2},\\
\dfrac{\pi (u_{i})\pi(v_{j})}{(\pi(u_{i})+\pi(u_{i}^{-1}))(N_{1}+N_{2})}: \;m_{1}+m_{2}+1\leq j.
\end{cases}$$
For $m_{1}+1\leq i\leq m_{1}+m_{2},$ we see that 
$$M_{i,j}:=
\begin{cases}
\dfrac{\pi (v_{i})\pi(u_{j}^{-1})}{\pi(v_{i})(N_{1}+N_{2})}: \;j\leq m_{1},\\
0: \;m_{1}+1\leq j\leq m_{1}+m_{2},\\
\dfrac{\pi (v_{i})\pi(v_{j})}{\pi(v_{i})(N_{1}+N_{2})}: \;m_{1}+m_{2}+1\leq j.
\end{cases}$$
For $m_{1}+m_{2}+1\leq i\leq m_{1}+2m_{2},$ we see that
$$M_{i,j}:=
\begin{cases}
\dfrac{\pi (v_{i})\pi(u_{j})}{\pi(v_{i})(N_{1}+N_{2})}: \;j\leq m_{1},\\
\dfrac{\pi(v_{i})\pi(v_{j})}{\pi(v_{i})(N_{1}+N_{2})}: \;m_{1}+1\leq j\leq m_{1}+m_{2},\\
0: \;m_{1}+m_{2}+1\leq j.\\
\end{cases}$$
We therefore have: 
$$M\vect{\xi}=\begin{pmatrix}
\dfrac{N_{1}\alpha +N_{2}\beta}{N_{1}+N_{2}}\vect{1}_{V_{1}}\\
\dfrac{N_{1}\alpha +N_{2}\beta}{N_{1}+N_{2}}\vect{1}_{V_{2}\sqcup V_{3}}
\end{pmatrix}.$$
Recall that $\vert V_{2}\vert=\vert V_{3}\vert.$ We calculate:
\begin{align*}
  &  \langle M\vect{\xi},\vect{\xi}\rangle=\begin{pmatrix}
\dfrac{N_{1}\alpha +N_{2}\beta}{N_{1}+N_{2}}\vect{1}_{V_{1}}\\
\dfrac{N_{1}\alpha +N_{2}\beta}{N_{1}+N_{2}}\vect{1}_{V_{2}\sqcup V_{3}}
\end{pmatrix}\cdot \begin{pmatrix}
\alpha\vect{1}_{V_{1}}\\
\beta\vect{1}_{V_{2}\sqcup V_{3}}\end{pmatrix}\\
&=\dfrac{N_{1}\alpha^{2}+N_{2}\alpha\beta }{N_{1}+N_{2}}\vert V_{1}\vert +\dfrac{N_{1}\alpha\beta +N_{2}\beta^{2}}{N_{1}+N_{2}}\vert V_{2}\sqcup V_{3}\vert\\
&=\dfrac{1}{N_{1}+N_{2}}\bigg[
N_{1}\alpha^{2} \vert V_{1}\vert+N_{2}\alpha\beta\vert V_{1}\vert+N_{1}\alpha \beta  \vert V_{2}\sqcup V_{3}\vert+N_{2}\beta^{2} \vert V_{2}\sqcup V_{3}\vert
\bigg]\\
&=\dfrac{N_{1}\alpha+N_{2}\beta}{N_{1}+N_{2}}\bigg[
\alpha \vert V_{1}\vert+\beta \vert V_{2}\sqcup V_{3}\vert\bigg]
\\&=0,
\end{align*}
since $$\vect{\xi}\cdot\vect{1}_{V}=\alpha \vert V_{1}\vert +\beta \vert V_{2}\sqcup V_{3}\vert=0.$$ 
Hence,
\begin{align*}
    \mu_{2}(M)=\langle M\vect{\phi},\vect{\phi}\rangle&=\langle M\vect{\xi},\vect{\xi}\rangle+\langle W_{1}M_{1}\vect{\psi},\vect{\psi}\rangle+\langle W_{2}M_{2}\vect{\psi},\vect{\psi}\rangle+\langle W_{3}M_{3}\vect{\psi},\vect{\psi}\rangle\\
    &\leq 0+o_{m_{1}}(1)+o_{m_{1}}(1)+o_{m_{1}}(1).
\end{align*}
\end{proof}

Applying Lemma \ref{lem: difference between adjacency matrix and expected adjacency matrix} to Lemmas \ref{lem: eigenvalue of Dpp m1=Omega(m2)}, \ref{lem: eigenvalue of Dpp m2=Omega(m1)}, and \ref{lem: eigenvalue of MDpp m1=theta(m2)}, we deduce the following.

\begin{lemma}\label{lem: eigenvalue of Dpp}
Let $m_{1},m_{2}\geq 1$. Suppose that $p$ satisfies $(m_{1}+2m_{2})p=\Omega_{m_{1}+2m_{2}}(\pi_{max}^{2}\log^{6}(m_{1}+2m_{2}))$. Let $\Sigma\sim \Dpp$. Then 
$$\lambda_{1}(\Sigma)\geq 1-o_{m_{1}+2m_{2}}(1)$$ with probability tending to $1$ as $m_{1}+2m_{2}$ tends to infinity.
\end{lemma}

%----------------------------------------------------------------------------------------
%   Property (T) in random groups
%----------------------------------------------------------------------------------------

%\input{Sections/Property_T_in_random_quotients}
\section{Property (T) in random quotients of hyperbolic groups}\label{sec: Property (T) in quotients of hyperbolic groups}
We can now turn to proving Theorem \ref{mainthm: property t in random quotients of hyperbolic groups}. We will first focus on a result for a slightly different model of random groups, where each relator is added independently with probability close to some $p$. We now fix $\omega$ a slow-growing function.

\begin{definition}
Given a word $w$, let $Path(w)$ be the set of paths in $C^{max}$ with label $w$. Note that for any word $w$, this set is of size at most $\vert V(C^{max})\vert.$
For a number $l-\omega (l)\leq L\leq l$, let $Sol_{l}(L)$ be the set of ordered triples $(a,b,c)$ with $a+b+c=L,l\slash 3-\omega(l)\slash 3 \leq a,b,c\leq l\slash 3+2$ (note that $\vert Sol_{l}(L)\vert \geq 1$ for all $l-\omega (l)\leq L\leq l$). 

The \emph{binomial model of random groups at length $\omega$-near $l$} is obtained as follows. Let $\Xi=\Xi(l)\leq \mu^{l-\omega(l)}.$ For each path $\gamma$ in $C^{max}$ of length between $l-\omega(l)$ and $l$, and each $(a,b,c)\in Sol(\vert \gamma\vert )$, add $lab(\gamma)$ to $R_{l;a,b,c}$ with probability $$\nu (\mathcal{Z}[\gamma,0])\Xi\slash (\omega(l)\slash 3)^{2}.$$ 

Let $\mathcal{R}_{l}=\sqcup_{a,b,c}\{x_{r}y_{r}z_{r}\;:\;r=(x_{r},y_{r},z_{r})\in R_{l;a,b,c}\},$ and consider $$G\slash \langle \langle \mathcal{R}_{l}\rangle\rangle.$$
\end{definition}

Note that $$Pr(w\in \mathcal{R}_{l})=(1+o_{l}(1))\vert Path(w)\vert\Xi \nu (\mathcal{Z}[w,0])\leq \vert V(C^{max})\vert K\slash \mu^{\vert w\vert}$$
for some constant $K>0.$
We can prove the following.
\begin{theorem}\label{thm: (T) in binomial hyperbolic quotients}
Let $G=\langle S\;\vert \; T\rangle $ be a finite presentation of a hyperbolic group with growth rate $\mathfrak{h}>0.$ Let $\mathcal{R}_{l}$ be obtained as above. If $\Xi\exp\{-\mathfrak{h}l\slash 3\}=\Omega_{l}(l^{7}),$ then $G\slash \langle\langle \mathcal{R}_{l}\rangle\rangle$ has Property (T) with probability tending to $1$ as $l$ tends to infinity.
\end{theorem}

This allows us to prove Theorem \ref{mainthm: property t in random quotients of hyperbolic groups}.

\begin{proof}[Proof of Theorem \ref{mainthm: property t in random quotients of hyperbolic groups}]
Consider the following model of random quotients of $G$, which we denote by $G_{1}(l,p)$. Let $p_{l}:\mathcal{L}^{G}\rightarrow [0,1]$. Add each word $w\in\mathcal{L}^{G}_{l,\omega (l)}$ to the relator set $\mathfrak{A}_{l}$ with probability $p_{l}(w)$, and consider $G\slash \langle \langle \mathfrak{A}_{l}\rangle \rangle$.

First note that, in the binomial model, there exists a constant $C$ such that $p(w\in \mathcal{R}_{l})\leq C\nu^{ld-\vert w\vert}$, independent of $w$. Therefore, as $(T)$ is preserved under quotients, Theorem \ref{thm: (T) in binomial hyperbolic quotients} implies that $G_{1}(l,p_{l})$ has (T) with probability $1-o_{l}(1)$ for all $p_{l}$ with $\min_{w\in\mathcal{L}^{G}_{l,\omega (l)}}\{p_{l}(w)\}\exp\{2\mathfrak{h}l\slash 3\}=\Omega_{l}(l^{7}).$

Fix $d>1\slash 3$ and choose $d>d_{1}>d_{2}>1\slash 3.$ Recall that Property (T) is preserved under adding relators. We have that $$\vert \mathcal{L}^{G}_{l}\vert =(1+o_{l}(1))\eta ' \exp\{\mathfrak{h}l\}$$ for some $\eta'>0$. We previously defined $$\Phi(g):=\left\vert\{w\in \mathcal{L}^{G}\;:\;\overline{w}=g\}\right \vert,$$ and proved that $\Phi(g)$ is uniformly bounded above by some constant $\Phi$. 

We refer to the model in Theorem \ref{mainthm: property t in random quotients of hyperbolic groups}  (i.e. killing a random subset $\mathfrak{B}_{l}\subseteq Ann_{l,\omega}(G)$ of size $\exp\{\mathfrak{h}ld\}$) by $G(l,d)$.  Finally, for $\rho_{l}: G\rightarrow [0,1]$ we produce the model $G_{2}(l,\rho)$ as follows. Add each element $g\in Ann_{l,\omega(l)}(G)$ to $X_{l} $ with probability $\rho_{l}(g)$, and consider $G\slash \langle \langle X_{l}\rangle \rangle$.

Now, consider $p_{l}$ defined by $$p_{l}(w):=\dfrac{\exp\{\mathfrak{h}(ld_{2}-\vert w\vert )\}}{\Phi(\overline{w})},$$ and the model $G_{1}(l,p)$. We may assume that $p_{l}(w,d_{2})=o_{l}(1)$ for all words $w$. The probability that an element $g$ appears in $\mathfrak{A}_{l}$ is 
\begin{align*}
    1-(1-p_{l}(w,d_{2}))^{\Phi (g)}&\approx (1+o_{l}(1))\Phi(g)p_{l}(w,d_{2})\\
    &=(1+o_{l}(1))\exp\{\mathfrak{h}(ld_{2}-\vert w\vert)\}.
\end{align*}

Therefore the model $G_{1}(l,p)$ is effectively the same as the model $G_{2}(l,\rho)$, where we add each element $g$ of $Ann_{l,\omega(l)}(G)$ to $Y_{l}$ with probability $\rho_{l}:=\exp\{\mathfrak{h}ld_{2}-\mathfrak{h}\vert g\vert \}$. In particular, if $d_{2}>1\slash 3,$ then by Theorem \ref{thm: (T) in binomial hyperbolic quotients}, we have that $G_{2}(l,\rho_{l})$ has property $(T)$ with overwhelming probability.

Now, it is easily seen that for $G\slash \langle \langle \mathfrak{B}_{l}\rangle\rangle\sim G(l,d),$ we have almost surely that for $0\leq n\leq \omega(l)$: $$\vert \mathfrak{B}_{l}\cap S_{l-n}\vert \geq (1+o_{l}(1))\exp\{\mathfrak{h}ld-\mathfrak{h}n\}\geq \exp\{\mathfrak{h}ld_{1}\}\;(*),$$ since $n\leq \omega (l)\leq \log\log l$.

Therefore, let $\Gamma=G\slash \langle \langle X_{l}\rangle \rangle \sim G_{2}(l,p_{l}(d_{2}))$: by the Chernoff bounds we have that $$\vert X_{l}\cap S_{l-n}\vert =(1+o_{l}(1))\exp\{\mathfrak{h}ld_{2}\}$$ for each $0\leq n\leq \omega (l).$

Choose a random subset $X_{l}\subset \mathfrak{B}_{l}$ of size $\exp\{\mathfrak{h}ld\}$: then $G\slash \langle \langle \mathfrak{B}_{l}\rangle\rangle\sim G(l,d)$, conditioned on $G$ having $(*)$. Since $\Gamma$ has property (T) with probability $1-o_{l}(1)$, and $G\sim G(l,d)$ has $(*)$ with probability $1-o_{l}(1),$ the result follows.
\end{proof}
It thus remains to prove Theorem \ref{thm: (T) in binomial hyperbolic quotients}. There is one small issue, in that we do not wish to study directed graphs. Given $G$ and $R_{l}$ as above, we can define the graph $\Psi(G,R_{l})$ by collapsing each pair of directed edges $(w,v^{-1}), (v^{-1},w)$ in $\Upsilon(G,\mathcal{L}^{G},R_{l})$ to a single undirected edge. In particular, the edge set of $\Psi(G,R_{l})$ is obtained by adding the undirected edge $\{w_{1},w_{3}^{-1}\}$ for each $r=(w_{1},w_{2},w_{3})\in R_{l}$. Since $L(\Upsilon(G,\mathcal{L}^{G},R_{l}))=L(\Psi(G,R_{l}))$, we can instead just analyse $\Psi(G,R_{l}).$

Let $\mathcal{L}^{+}_{l\slash 3,\omega}$ be the set of $w\in\mathcal{L}^{G}_{l\slash 3,\omega(l)\slash 3}$ with $w^{-1}\notin \mathcal{L}^{G}$. Define the set $$\mathcal{L}^{-}_{l\slash 3,\omega }:=\{w\in\mathcal{L}^{G}_{l\slash 3,\omega (l)\slash 3}\;:\;w^{-1}\in \mathcal{L}^{+}_{l\slash 3,\omega(l)\slash 3}\}.$$ Finally, let $\mathcal{L}^{+-}_{l\slash 3,\omega(l)\slash 3}$ be the set $$\mathcal{L}^{+-}_{l\slash 3,\omega}:=\{w\in \mathcal{L}^{G}\;:\;w^{-1}\in\mathcal{L}^{G}_{l\slash 3,\omega\slash 2}\}.$$
Note that, if $w=w^{-1}$, then $w$ is the trivial word. We now apply our previous work on geodesics in hyperbolic groups to count the degrees of vertices in our random graphs. Recall the Chernoff bounds: for $X\sim Bin(m,p)$ and $\delta\in [0,1],$ $$\mathbb{P}(\vert X - mp\vert \geq \delta mp)\leq 2\exp(-mp \delta^{2}\slash 3).$$ We now prove the following.
\begin{lemma}
\label{lem: degrees l=0 mod 3}
Let $G=\langle S\;\vert \; T\rangle $ be a hyperbolic group with growth rate $\mathfrak{h}>0.$  Let $R_{l}$ be obtained as in Theorem \ref{thm: (T) in binomial hyperbolic quotients}, where $p\exp\{-\mathfrak{h}l\slash 3\}=\Omega_{l}(l)$. Then with probability tending to $1$ as $l$ tends to infinity, for any $v\in \mathcal{L}^{+}_{l,\omega (l)\slash 3},w\in \mathcal{L}^{+-}_{l,\omega (l)\slash 3}$.
\begin{align*}
rep(v)
   &=(1+o_{l}(1))\Xi\nu(\mathcal{Z}[v]),\\
 rep(w)
   &=(1+o_{l}(1))\Xi(\nu(\mathcal{Z}[w])+\nu(\mathcal{Z}[w^{-1}])).
\end{align*}
\end{lemma}
\begin{proof}
Let us suppose that $w\in \mathcal{L}^{+-}_{l\slash 3,\omega (l)\slash 3}$. Then we calculate, using Lemma \ref{lem: measure of sets middle path}.

\begin{align*}
    \mathbb{E}(rep(w))&=\sum \limits_{\substack{\sigma\gamma\sigma' \mbox{a path }\\lab(\gamma)=w}} Pr(\sigma\gamma\sigma'\mbox{ added to } R_{\vert \sigma\vert,\vert w\vert,\vert\sigma'\vert})\\&\hspace*{30 pt}+\sum \limits_{\substack{\sigma\gamma\sigma' \mbox{a path }\\lab(\gamma)=w^{-1}}} Pr(\sigma\gamma\sigma'\mbox{ added to } R_{\vert \sigma\vert,\vert w^{-1}\vert,\vert\sigma'\vert})\\
    &=\Xi\nu (\mathcal{Z}[w])+\nu (\mathcal{Z}[w^{-1}]).
\end{align*}
The result follows by an application of the Chernoff bounds.
\end{proof}

We next need to analyse edge probabilities.
\begin{lemma}
\label{lem: edge probs l=0 mod 3}
Let $G=\langle S\;\vert \; T\rangle $ be a hyperbolic group with growth rate $\mathfrak{h}>0.$  Let $l\geq 1$. Let $R_{l}$ be obtained as in Theorem \ref{thm: (T) in binomial hyperbolic quotients}, where $\Xi \exp\{-\mathfrak{h}l\slash 3\}=\Omega_{l}(l)$. Let $E$ be the edge set of $\Psi(G,\mathfrak{R}_{l})$. There exists a constant $\pi_{min}>0$ and a function $\pi:L^{G}\rightarrow [\pi_{min},\infty)$ such that for any $u\in \mathcal{L}^{+}_{l\slash 3,\omega (l)\slash 3},v\in\mathcal{L}^{-}_{l\slash 3,\omega (l)\slash 3},w,w'\in \mathcal{L}^{+-}_{l\slash 3,\omega (l)\slash 3}$:
\begin{align*}
 Pr(\{u,v\}\in E)&=(1+o_{l}(1))\pi(u)\pi(v^{-1})\exp\{-2\mathfrak{h}l\slash3 \}3\Xi\slash \omega (l),\\
 Pr(\{u,w\}\in E)&=(1+o_{l}(1))\pi(u)\pi(w^{-1})\exp\{-2\mathfrak{h}\slash3 \}3\Xi\slash \omega(l),\\
  Pr(\{v,w\}\in E)&=(1+o_{l}(1))\pi(v^{-1})\pi(w)\exp\{-2\mathfrak{h}l\slash 3\}3\Xi\slash \omega(l),\\
   Pr(\{w,w'\}\in E)&=(1+o_{l}(1))[\pi(w)\pi(w'^{-1})+\pi(w^{-1})\pi(w')]\exp\{-2\mathfrak{h}l\slash 3\}3\Xi\slash \omega(l).
\end{align*}

\end{lemma}
Clearly $u$ is never connected to $v^{-1}$, etc.
\begin{proof}
We prove Statement $4$ only. For a path $\gamma=(e_{i_{1}},\hdots ,e_{i_{n}})$, let $$\pi (\gamma) =\alpha_{i_{1}}\beta_{i_{n}} \exp\{\mathfrak{h}(l\slash 3 -n)\}.$$ Define $$\pi (w)=\sum_{lab(\gamma)=w}\pi (\gamma).$$ The edge $\{w,w'\}$ can occur when either $w uw'^{-1}$ is added to $R_{l;\vert w\vert,\vert u\vert ,\vert w'\vert}$, or $w' uw^{-1}$ is added to $R_{l;\vert w'\vert,\vert u\vert ,\vert w\vert}$. By Lemma \ref{lem: measure of sets between two paths}, the probability such a word appears in the correct $R_{l;a,b,c}$ is $$(1+o_{l}(1))\dfrac{\omega (l)}{3}(\pi (w)\pi (w'^{-1})+\pi (w^{-1})\pi(w')) \Xi\exp\{-2\mathfrak{h}l\slash 3\}(\omega(l)\slash 3)^{-1}.$$
The probability the word is then added to the correct $R_{l;a,b,c}$ to obtain the desired edge is $(\omega(l)\slash 3)^{-2}$, and the result follows.
\end{proof}
Finally, we can easily observe the following.
\begin{lemma}
Let $G=\langle S\;\vert \; T\rangle $ be a hyperbolic group with growth rate $\mathfrak{h}>0.$  Let $l\geq 1$, and let $R_{l}$ be obtained as in Theorem \ref{thm: (T) in binomial hyperbolic quotients}. Then for all $g\in G$ with $g\in V$:
$$deg(g)=(1+o_{l}(1))deg(g^{-1})=(1+o_{l}(1))rep(g).$$
\end{lemma}

We now note that, importantly, for small $p$, there are not too many double edges in $\Psi(G,R_{l}).$ A \emph{matching} in a graph is a set, $\mathcal{M}$, of edges in the graph such that no two edges in $\mathcal{M}$ share a common endpoint.
\begin{lemma}\label{lem: double edges form matching} (c.f. \cite{antoniuktriangle})
Let $G=\langle S\;\vert \; T\rangle $ be a hyperbolic group, with growth rate $\mathfrak{h}>0.$ Let $R_{l}$ be constructed as in Theorem \ref{thm: (T) in binomial hyperbolic quotients}. If $\Xi\leq \exp\{\mathfrak{h}ld\}$ for some $d<5\slash 12$, then with probability tending to $1$ as $l$ tends to infinity, the set of double edges in $\Psi(G,R_{l})$ forms a matching, and there are no triple edges in $\Psi(G,R_{l})$ .
\end{lemma}
The proof of the above is effectively that of \cite[Lemma 6.3]{ashcroft2021property}, which closely follows the \cite[p.12]{antoniuktriangle}. 
\begin{proof}
Fix $\Xi\leq \exp\{\mathfrak{h}ld\}$ for some $d<5\slash 12$.
The probability, $\mathbb{P}_{3}$, that there exists a pair of vertices $u,v$ with at least three edges between $u$ and $v$ is bounded above by 
\begin{align*}
        \mathbb{P}_{3}&\leq O_{l}\left((\exp\{\mathfrak{h}l\slash 3\})^{2}(\exp\{\mathfrak{h}l\slash 3\})^{3}\Xi^{3}\exp\{-2\mathfrak{h}l\}\log l\right)\\
        &\leq O_{l}\left(\exp\{\mathfrak{h}l\slash 3\}\right)^{(5+9d-9)\mathfrak{h}l}2\log l)\\
        &=o_{l}(1),
    \end{align*}
since $d<5\slash 12<4\slash 9$. The probability, $\mathbb{P}_{doub}$ that there are vertices $u,v,w$ with double edges between $u$ and $v$ and $u$ and $w$ is bounded by
\begin{align*}
        \mathbb{P}_{doub}&=     O_{l}(\left (\exp\{\mathfrak{h}l\slash 3\})^{3}(\eta\exp\{\mathfrak{h}l\slash 3\})^{4}\Xi^{4}\exp\{-4\mathfrak{h}l\}\log l\right)\\
        &\leq O_{l}(\left(\exp\{\mathfrak{h}l\slash 3\}\right)^{(12d-5)}\log l)\\
            &=o_{l}(1),
    \end{align*}
    since $d<5\slash 12.$
\end{proof}
This allows us to prove Theorem \ref{thm: (T) in binomial hyperbolic quotients}.
\begin{proof}[Proof of Theorem \ref{thm: (T) in binomial hyperbolic quotients}]
As (T) is an increasing property, we may assume that $\Xi$ satisfies the conditions of Theorem \ref{thm: (T) in binomial hyperbolic quotients} and Lemma \ref{lem: double edges form matching}. For example, we may take $\Xi=\exp\{\mathfrak{h}ld\}$ for $1\slash 3 <d<5\slash 12$. Note that the minimum degree of $\Psi(G,\mathfrak{R}_{l})$ tends to infinity with probability tending to $1$. Therefore, as the set of double edges forms a matching in $\Psi(G,\mathfrak{R}_{l})$, by collapsing double edges in $\Psi(G,\mathfrak{R}_{l})$, we find a graph $\Psi'_{l}$ with at most one edge between any two vertices such that $$\lambda_{1}(\Psi'_{l})=\lambda_{1}(\Psi(G,\mathfrak{R}_{l}))+o_{l}(1)$$
almost surely. Applying Lemma \ref{lem: edge probs l=0 mod 3}, we observe that $$\Psi'\sim\mathring{\Delta}\left(\vert \mathcal{L}^{+-}_{l\slash 3,\omega (l)\slash 3}\vert,\vert \mathcal{L}^{+}_{l\slash 3,\omega (l)\slash 3}\vert,\pi ,\Xi\exp\{-2\mathfrak{h}l\slash 3\}\slash \omega(l),*\right),$$
where $*:V_{1}\rightarrow V_{1}$ maps $w\mapsto w^{-1}$ (and hence is fixed point free). Now, $\vert \mathcal{L}^{G}_{l,\omega(l)\slash 3}\vert = \vert \mathcal{L}^{+-}_{l\slash 3,\omega(l)\slash 3}\vert+2\vert \mathcal{L}^{+}_{l\slash 3,\omega(l)\slash 3}\vert$, $\log^{6}(\vert \mathcal{L}^{G}_{l\slash 3,\omega(l)\slash 3}\vert)=O_{l}(l^{6})$. Furthermore, $\pi_{max}\leq \log\log l$. As $\Xi$ satisfies the assumptions of Theorem \ref{thm: (T) in binomial hyperbolic quotients}, we may apply Lemma \ref{lem: eigenvalue of Dpp}, to deduce that with probability $1-o_{l}(1)$; $$\lambda_{1}(\Psi')\geq 1+o_{l}(1).$$ Since $$\lambda_{1}(\Psi(G,\mathfrak{R}_{l}))=\lambda_{1}(\Psi'(G,\mathfrak{R}_{l}))+o_{l}(1)$$ almost surely, we have $$\lambda_{1}(\Psi(G,\mathfrak{R}_{l}))\geq 1+o_{l}(1)$$ almost surely. Furthermore, with probability $1-o_{l}(1)$, $\mathcal{L}^{G}_{l\slash 3}\subseteq \mathcal{W}(R_{l})$, and so $\mathcal{W}(R_{l})$ generates a finite index subgroup of $G$ by Lemma \ref{lem: Cmax gens finite index}. The result now follows by Corollary \ref{cor: directed new spectral criterion} and Lemma \ref{lem: degrees l=0 mod 3}.
\end{proof}

We now very briefly indicate how to alter the preceding proofs to obtain a proof of Theorem \ref{mainthm: property t in random quotients of hyperbolic groups aperiodic case}. Firstly, we can replace Lemma \ref{lem: convergence of approximation} by the following classical result in ergodic Markov chains.
\begin{lemma}\label{lem: convergence of approximation aperiodic}
Let $B$ be the edge incidence matrix of $C^{max}$, with left (right) $\mu$ eigenvector $\alpha$ ($\beta)$. If $C^{max}$ is aperiodic, we have convergence $$\dfrac{B^{m}}{\mu^{m}}\rightarrow \beta\alpha.$$
\end{lemma}
We then switch to a new binomial model as follows.  

\begin{definition}

The \emph{binomial model of random groups at length $3l$} is obtained as follows. Let $\Xi=\Xi(l)\leq \mu^{l-\omega(l)}.$ For each path $\gamma$ in $C^{max}$ of length $3l$, add $lab(\gamma)$ to $R_{3l;l,l,l}$ with probability $$\nu (\mathcal{Z}[\gamma,0])\Xi,$$ 

and consider $$G\slash \langle \langle \mathcal{R}_{l}\rangle\rangle.$$
\end{definition}

The aim is to prove the following.
\begin{theorem}\label{thm: binomial (T) aperiodic case}
Let $G=\langle S\;\vert \; T\rangle $ be a finite presentation of a hyperbolic group with growth rate $\mathfrak{h}>0,$ and $C^{max}$ aperiodic, Let $\mathcal{R}_{3l}$ be obtained as above. If $\Xi\exp\{-\mathfrak{h}l\}=\Omega_{l}(l^{7}),$ then $G\slash \langle\langle \mathcal{R}_{3l}\rangle\rangle$ has Property (T) with probability tending to $1$ as $l$ tends to infinity.
\end{theorem}

The proof of the above follows immediately as the proof of Theorem \ref{thm: (T) in binomial hyperbolic quotients}, using the following lemma, which is obtained by an application of Lemma \ref{lem: convergence of approximation aperiodic}.
\begin{lemma}\label{lem: measure of starting path aperiodic}
Let $\gamma=(e_{i_{1}},\hdots ,e_{i_{l}}),\gamma'=(e_{i'_{1}},\hdots ,e_{i'_{l}}),$ be paths in $C^{max}$. The following hold.
\begin{align*}
    &\sum \limits_{\sigma=(e_{k_{1}},\hdots e_{k_{l}})}\nu (\mathcal{Z}[\gamma,\sigma,\gamma'],0)=(1+o_{L}(1)) \alpha_{i_{1}}\beta_{i_{m}}\alpha_{j_{1}}\beta_{j_{n}}\slash \mu^{m+n-1}.\\
 &\sum \limits_{\sigma=(e_{j_{1}},\hdots ,e_{j_{l}})}\sum \limits_{\sigma'=(e_{k_{1}},\hdots e_{k_{l}})}\nu (\mathcal{Z}[\gamma,\sigma,\sigma'],0)=\nu (\mathcal{Z}[\gamma,0])\\
 & \sum \limits_{\sigma=(e_{j_{1}},\hdots ,e_{j_{l}})}\sum \limits_{\sigma'=(e_{k_{1}},\hdots e_{k_{l}})}\nu (\mathcal{Z}[\sigma,\sigma',\gamma],0)=\nu (\mathcal{Z}[\gamma,0])\\
 &  \sum \limits_{\sigma=(e_{j_{1}},\hdots ,e_{j_{l}})}\sum \limits_{\sigma'=(e_{k_{1}},\hdots e_{k_{l}})}\nu (\mathcal{Z}[\sigma,\gamma,\sigma'],0)=\nu(\mathcal{Z}[\gamma,0]).
\end{align*}

\end{lemma}

We then prove Theorem \ref{mainthm: property t in random quotients of hyperbolic groups aperiodic case} from Theorem \ref{thm: binomial (T) aperiodic case} in a manner similar to that employed to deduce Theorem \ref{mainthm: property t in random quotients of hyperbolic groups} from Theorem \ref{thm: (T) in binomial hyperbolic quotients}. The main observation is that for a choice $T_{l}\subseteq Ann_{l,2}(G)$ of size at least $\exp\{\mathfrak{h}ld\}$ for any $\epsilon>0$, we have $\vert T_{l}\cap S_{3\lfloor l\slash 3\rfloor} \vert \geq \exp\{\mathfrak{h}l(d-\epsilon)\}$ almost surely. Theorem \ref{thm: binomial (T) aperiodic case} then implies that the group $G\slash \langle \langle T_{l}\cap S_{3\lfloor l\slash 3\rfloor} \rangle \rangle$ has (T) almost surely.

	\bibliographystyle{alpha}
	\bibliography{bib}

\begin{thebibliography}{BdlHV08}

\bibitem[Ago13]{Agol13}
Ian Agol.
\newblock The virtual {H}aken conjecture.
\newblock {\em Doc. Math.}, 18:1045--1087, 2013.
\newblock With an appendix by Agol, Daniel Groves, and Jason Manning.

\bibitem[A{\L}S15]{antoniuktriangle}
Sylwia Antoniuk, Tomasz {\L}uczak, and Jacek \'{S}wi\c{a}tkowski.
\newblock Random triangular groups at density 1/3.
\newblock {\em Compos. Math.}, 151(1):167--178, 2015.

\bibitem[ARD20]{ashcroftrandom}
Calum~J. Ashcroft and Colva~M. Roney-Dougal.
\newblock On random presentations with fixed relator length.
\newblock {\em Comm. Algebra}, 48(5):1904--1918, 2020.

\bibitem[Ash21]{ashcroft2021property}
Calum~J Ashcroft.
\newblock Property {(T)} in density-type models of random groups.
\newblock {\em arXiv preprint arXiv:2104.14986}, 2021.

\bibitem[BdlHV08]{bekkadelaharpe}
Bachir Bekka, Pierre de~la Harpe, and Alain Valette.
\newblock {\em Kazhdan's {P}roperty ({T})}, volume~11 of {\em New Mathematical
  Monographs}.
\newblock Cambridge University Press, Cambridge, 2008.

\bibitem[Bha97]{BhatiaRajendra1997Ma/R}
Rajendra Bhatia.
\newblock {\em Matrix analysis / Rajendra Bhatia.}
\newblock Graduate texts in mathematics ; 169. Springer, New York ; London,
  1997.

\bibitem[Cal13]{calegari2013ergodic}
Danny Calegari.
\newblock The ergodic theory of hyperbolic groups.
\newblock {\em Contemp. Math}, 597:15--52, 2013.

\bibitem[Can84]{Cannon_combinatorial_structure_hyperbolic_groups}
James~W Cannon.
\newblock The combinatorial structure of cocompact discrete hyperbolic groups.
\newblock {\em Geometriae Dedicata}, 16(2):123--148, 1984.

\bibitem[CLV04]{chungrandomgraph}
Fan Chung, Linyuan Lu, and Van Vu.
\newblock The spectra of random graphs with given expected degrees.
\newblock {\em Internet Math.}, 1(3):257--275, 2004.

\bibitem[DM19]{DRUTU_Mackay}
Cornelia Druţu and John~M. Mackay.
\newblock Random groups, random graphs and eigenvalues of p-laplacians.
\newblock {\em Advances in Mathematics}, 341:188--254, 2019.

\bibitem[Duo17]{duong}
Yen Duong.
\newblock {\em On {R}andom {G}roups: {T}he {S}quare {M}odel at {D}ensity d <
  1/3 and as {Q}uotients of {F}ree {N}ilpotent {G}roups}.
\newblock ProQuest LLC, Ann Arbor, MI, 2017.
\newblock Thesis (Ph.D.)--University of Illinois at Chicago.

\bibitem[FK15]{frieze_karonski}
Alan Frieze and Michal Karonski.
\newblock {\em Introduction to Random Graphs}.
\newblock Cambridge University Press, 2015.

\bibitem[FW21]{futer2021cubulating}
David Futer and Daniel~T Wise.
\newblock Cubulating random quotients of hyperbolic cubulated groups.
\newblock {\em arXiv preprint arXiv:2106.04497}, 2021.

\bibitem[Gro87]{Gromov_hyperbolic}
Mikhael Gromov.
\newblock Word hyperbolic groups.
\newblock In {\em S. M. Gersten, editor, Essays in Group Theory}, volume~8 of
  {\em Mathematical Sciences Research Institute Publications}, pages 75--264.
  Springer-Verlag, 1987.

\bibitem[Gro93]{gromovasymptotic}
M.~Gromov.
\newblock Asymptotic invariants of infinite groups.
\newblock In {\em Geometric group theory, {V}ol. 2 ({S}ussex, 1991)}, volume
  182 of {\em London Math. Soc. Lecture Note Ser.}, pages 1--295. Cambridge
  Univ. Press, Cambridge, 1993.

\bibitem[GTT18]{gekhtman2018counting}
Ilya Gekhtman, Samuel~J Taylor, and Giulio Tiozzo.
\newblock Counting loxodromics for hyperbolic actions.
\newblock {\em Journal of Topology}, 11(2):379--419, 2018.

\bibitem[GTT20]{gekhtman2020central}
Ilya Gekhtman, Samuel~J Taylor, and Giulio Tiozzo.
\newblock Central limit theorems for counting measures in coarse negative
  curvature.
\newblock {\em arXiv preprint arXiv:2004.13084}, 2020.

\bibitem[HJ94]{hornjohnson}
Roger~A. Horn and Charles~R. Johnson.
\newblock {\em Topics in matrix analysis}.
\newblock Cambridge University Press, Cambridge, 1994.
\newblock Corrected reprint of the 1991 original.

\bibitem[KK13]{kotowski}
Marcin Kotowski and Micha\l\; Kotowski.
\newblock Random groups and {P}roperty {$(T)$}: \.{Z}uk's theorem revisited.
\newblock {\em J. Lond. Math. Soc. (2)}, 88(2):396--416, 2013.

\bibitem[Mon21]{montee2021random}
MurphyKate Montee.
\newblock Random groups at density $d<3/14$ act non-trivially on a {CAT}(0)
  cube complex.
\newblock {\em arXiv preprint arXiv:2106.14931}, 2021.

\bibitem[Mon22]{Montee_prop_t}
MurphyKate Montee.
\newblock Property {(T)} in k-gonal random groups.
\newblock {\em Glasgow Mathematical Journal}, page 1–5, 2022.

\bibitem[MP15]{mackay_przytycki2015balanced}
John~M Mackay and Piotr Przytycki.
\newblock Balanced walls for random groups.
\newblock {\em The Michigan Mathematical Journal}, 64:397--419, 2015.

\bibitem[Odr16]{odrsquaremodel}
Tomasz Odrzyg\'{o}\'{z}d\'{z}.
\newblock The square model for random groups.
\newblock {\em Colloq. Math.}, 142(2):227--254, 2016.

\bibitem[Odr18]{odrcubulatingsquare}
Tomasz Odrzyg\'{o}\'{z}d\'{z}.
\newblock Cubulating random groups in the square model.
\newblock {\em Israel J. Math.}, 227(2):623--661, 2018.

\bibitem[Odr19]{odrzygozdz2019bent}
Tomasz Odrzyg{\'o}{\'z}d{\'z}.
\newblock Bent walls for random groups in the square and hexagonal model.
\newblock {\em arXiv preprint arXiv:1906.05417}, 2019.

\bibitem[Oli09]{oliveira2009concentration}
Roberto~Imbuzeiro Oliveira.
\newblock Concentration of the adjacency matrix and of the laplacian in random
  graphs with independent edges.
\newblock {\em arXiv preprint arXiv:0911.0600}, 2009.

\bibitem[Oll04]{olliviersharp}
Y.~Ollivier.
\newblock Sharp phase transition theorems for hyperbolicity of random groups.
\newblock {\em Geom. Funct. Anal.}, 14(3):595--679, 2004.

\bibitem[Oll05]{Ollivier_Jan_Invitation}
Yann Ollivier.
\newblock {\em A {J}anuary 2005 invitation to random groups}, volume~10 of {\em
  Ensaios Matem\'{a}ticos [Mathematical Surveys]}.
\newblock Sociedade Brasileira de Matem\'{a}tica, Rio de Janeiro, 2005.

\bibitem[OW11]{Ollivier-Wise}
Y.~Ollivier and D.T. Wise.
\newblock Cubulating random groups at density less than 1/6.
\newblock {\em Transactions of the American Mathematical Society},
  363(9):4701--4733, 2011.

\bibitem[Oza16]{ozawapropertyt}
Narutaka Ozawa.
\newblock Noncommutative real algebraic geometry of {K}azhdan's {P}roperty
  ({T}).
\newblock {\em J. Inst. Math. Jussieu}, 15(1):85--90, 2016.

\bibitem[Par64]{parry1964intrinsic}
William Parry.
\newblock Intrinsic {M}arkov chains.
\newblock {\em Transactions of the American Mathematical Society},
  112(1):55--66, 1964.

\bibitem[Tro12]{tropp2012user}
Joel~A Tropp.
\newblock User-friendly tail bounds for sums of random matrices.
\newblock {\em Foundations of computational mathematics}, 12(4):389--434, 2012.

\bibitem[Wey12]{Weyl}
Hermann Weyl.
\newblock Das asymptotische {V}erteilungsgesetz der {E}igenwerte linearer
  partieller {D}ifferentialgleichungen (mit einer {A}nwendung auf die {T}heorie
  der {H}ohlraumstrahlung).
\newblock {\em Math. Ann.}, 71(4):441--479, 1912.

\bibitem[\.{Z}96]{zuk1996}
A.~\.{Z}uk.
\newblock La propri{\'e}t{\'e} {(T)} de {K}azhdan pour les groupes agissant sur
  les polyedres.
\newblock {\em CR Acad. Sci. Paris S{\'e}r. I Math.}, 323:453--458, 1996.

\bibitem[\.{Z}03]{Zuk}
A.~\.{Z}uk.
\newblock {P}roperty {(T)} and {K}azhdan constants for discrete groups.
\newblock {\em Geom. Funct. Anal.}, 13(3):643--670, 2003.

\end{thebibliography}
	{\sc{DPMMS, Centre for Mathematical Sciences, Wilberforce Road, Cambridge, CB3 0WB, UK}.} \textit{\textsc{\textit{E-mail address}}:}\textsc{ cja59@dpmms.cam.ac.uk}
\end{document}